\let\ORIlabel\label
\let\ORIrefstepcounter\refstepcounter
\AddToHook{package/hyperref/before}{
   \let\label\ORIlabel 
   \let\refstepcounter\ORIrefstepcounter}
\documentclass[final,onefignum,onetabnum,final]{siamart220329}

\usepackage{amssymb,amsfonts, amsmath} % 
\usepackage{bm}
\usepackage{todonotes}
\usepackage{mathtools, relsize,geometry}
\usepackage{mathrsfs}
\usepackage{multirow}
\newsiamremark{remark}{Remark}

\makeatletter

\renewcommand{\vec}[1]{%
	\ifcat\relax\noexpand#1%
	\ensuremath{\boldsymbol{\lowercase{#1}}}%
	\else
	\ensuremath{\mathbf{\lowercase{#1}}}%
	\fi
}
\newcommand{\bF}{\bm F}
\newcommand{\bG}{\bm G}

\newtheorem{assumption}[theorem]{Assumption}

\newcommand{\rev}[1]{{\color{black}#1}}
\newcommand{\DP}[1]{{\color{black}#1}}

\usepackage{pgfplots}
\usepackage{graphicx,epstopdf} 
\usepackage{subfig}
\usepackage{tikz}
\usepackage{algorithm}
\usepackage{algpseudocode}
\usepackage{algorithmicx}

\usetikzlibrary{shapes.geometric}

\definecolor{matlabred}{rgb}{0.9047,    0.1918,    0.1988}
\definecolor{matlabblue}{rgb}{0.2941    0.5447    0.7494}
\definecolor{matlabgreen}{rgb}{	0.3718    0.7176    0.3612}
\definecolor{matlaborange}{rgb}{1.0000    0.5482    0.1000}

\usepackage{algpseudocode}
\algnewcommand{\LineComment}[1]{\Statex \(\%\) \small \textit{#1} \(\%\)}

\Crefname{ALC@unique}{Line}{Lines}
\usepackage{algorithm}
\DeclareMathOperator{\argmin}{\rm arg\,min}

\title{Residual Recombination Methods as Anderson-like Acceleration:
An Algebraic Interpretation of BoostConv\thanks{Version of \today}}

\author{Vincenzo Citro\thanks{DIIN, University of Salerno, Via Giovanni Paolo II, 84084 Fisciano (SA), Italy,
{\tt vcitro@unisa.it}} \and Davide Palitta\thanks{Dipartimento di Matematica and (AM)$^2$,
Alma Mater Studiorum Universit\`a di Bologna,
Piazza di Porta San Donato  5, I-40127 Bologna, Italy,
{\tt davide.palitta@unibo.it}} }

\begin{document}
\maketitle

\renewcommand{\thefootnote}{\fnsymbol{footnote}}
\maketitle \pagestyle{myheadings} \thispagestyle{plain}
\markboth{ V.\ CITRO, D.\ PALITTA}{BOOSTCONV AS ANDERSON-LIKE ACCELERATION}

%% ------------------------------------------------------------------
%% ABSTRACT
%% ------------------------------------------------------------------

\begin{abstract}
{\tt BoostConv} has been introduced in earlier works as an effective acceleration technique for nonlinear iterative processes and has been successfully employed in a variety of applications to enhance convergence rates or to compute unstable fixed points that are otherwise inaccessible through standard approaches. Despite its demonstrated practical effectiveness, the theoretical properties of the method have not yet been fully characterized.

In this work, we present a more robust formulation of the {\tt BoostConv} algorithm and, for the first time, provide a rigorous proof of its \rev{local} convergence. The proposed analysis places {\tt BoostConv} within a precise mathematical framework, clarifying its interpretation as a nonlinear convergence accelerator and establishing sufficient conditions under which convergence to a fixed point is guaranteed.

The theoretical findings are illustrated through several numerical examples, spanning from a linear problem to a low-dimensional benchmark and a large-scale incompressible Navier–Stokes simulation. These results demonstrate the robustness and practical relevance of the proposed method and bridge the gap between empirical performance and rigorous analysis, paving the way for further developments and applications to complex nonlinear problems.
\end{abstract}



%%%%%%%%%%%%%%%%%%%%%%%%%%%%%%%%%%%%%%%%%%%%%%%%%%

\section{Introduction}

The numerical solution of large-scale nonlinear systems
\begin{equation}
\label{eq:nonlinear_system}
\bF(x)=0, \qquad \bF:\mathbb{R}^n \to \mathbb{R}^n,
\end{equation}
is a central problem in scientific computing, arising in a wide range of applications
including computational fluid dynamics, nonlinear partial differential equations,
optimization, and inverse problems, to name a few. When the dimension $n$ of the problem at hand is large, the efficiency
and robustness of the nonlinear solver play a decisive role in the overall feasibility
of the simulation.

A classical approach to \eqref{eq:nonlinear_system} consists in applying Newton's method,
which at iteration $k$ requires the solution of the linearized system
\begin{equation}
\label{eq:newton_linearized}
\frac{\partial \bF}{\partial x}(x_k)\,(x_{k+1}-x_k) = -\bF(x_k), \,\, k \ge 0,
\end{equation}
where $x_k$ is the current approximate solution that follows a given initial guess $x_0$.
While this strategy enjoys quadratic local convergence, its practical applicability is
often limited by the cost of assembling, storing, and inverting the Jacobian matrix,
especially in large-scale or matrix-free settings \cite{dennis_schnabel,nocedal_wright}.
For this reason, a vast literature has been devoted to the development of approximate
Newton methods in which the action of the inverse Jacobian is replaced by a cheaper
operator.

A common abstraction of these approaches is obtained by rewriting
\eqref{eq:newton_linearized} in the iterative form
\begin{equation}
\label{eq:basic_iteration}
x_{k+1} = x_k + B r_k,
\qquad r_k = -\bF(x_k),
\end{equation}
where $B$ is a linear operator approximating $(\partial \bF/\partial x)^{-1}$.
Depending on the choice of $B$, \eqref{eq:basic_iteration} may represent a fixed-point
iteration, a quasi-Newton method, or a single step of a Jacobian-free Newton--Krylov
procedure \cite{kelley_book}. The convergence properties of the method
are therefore strongly influenced by the quality of the approximation $B$ and by the
spectral features of the underlying problem. \rev{More generally, the operator $B$ may be allowed to change along the iteration, i.e., $B=B_k$, as it happens, e.g., in quasi-Newton methods; cf. Remark~\ref{remark_broyden}. For the sake of simplicity, in this paper we consider a fixed $B$.}

In many applications, the iteration \eqref{eq:basic_iteration} converges only slowly
or may even fail to converge due to the presence of weakly stable or unstable modes.
This phenomenon is particularly pronounced in nonlinear dynamical systems arising
from the discretization of partial differential equations, where the Jacobian is
often non-normal and characterized by a small number of dominant modes
\cite{saad_iterative,saad_szyld}. In such situations, significant gains can be achieved
by augmenting the basic iteration with acceleration techniques that exploit information
from previous residuals.

Residual recombination and multisecant strategies provide a powerful framework in this
direction. Classical quasi-Newton methods, beginning with Broyden's updates
\cite{broyden1965}, construct low-rank corrections to the operator $B$ so as to enforce
secant conditions derived from past iterates. Closely related ideas underlie Anderson
acceleration and mixing methods, which have recently received renewed attention and
for which rigorous convergence results are now available
\cite{anderson1965,EvansEtAl2020,TothKelley2015,WalkerNi2011,PollockRebholz2023,PollockRebholzBook}. \rev{In particular, for Anderson acceleration it has been shown that the accelerated scheme not only converges, but it also improves the convergence rate of the underlying fixed-point iteration and can enlarge its domain of convergence; see, e.g.,~\cite{EvansEtAl2020,PollockRebholz2023,PollockRebholzBook}.}

Within this landscape, the {\tt BoostConv} algorithm introduced in \cite{CITRO2017234}
stands out as an effective and remarkably simple residual recombination technique.
{\tt BoostConv} modifies the iteration \eqref{eq:basic_iteration} by replacing the current
residual $r_k$ with a linear combination of past residual differences, obtained by
solving a small least-squares problem. The resulting method can be interpreted as a
stabilization and acceleration mechanism that selectively targets the slowly converging
components of the iteration, while leaving the original solver essentially unchanged.
A distinctive feature of {\tt BoostConv} is that it can be implemented as a black-box wrapper
around an existing solver, requiring only minimal modifications to the original code.

Originally proposed and validated in the context of computational fluid dynamics, {\tt BoostConv} has been shown to effectively accelerate convergence and to enable the computation of unstable steady states in large-scale simulations \cite{citroPOFBoostconv}. {\color{black}While its practical effectiveness has been clearly demonstrated, a rigorous theoretical characterization of its convergence properties is still lacking.}

The aim of this paper is to address these aspects by presenting a more robust formulation of the {\tt BoostConv} algorithm and by providing, for the first time, a rigorous proof of \rev{local} convergence under suitable assumptions. By framing {\tt BoostConv} within the theoretical setting of Anderson acceleration, we analyze its convergence properties and illustrate the theoretical findings through some representative numerical examples. The results confirm that {\tt BoostConv} constitutes a robust and efficient strategy for accelerating nonlinear iterations of the form \eqref{eq:basic_iteration}, while incurring only minimal additional computational cost.

The remainder of the paper is organized as follows. In section~\ref{Background material} we report some background material needed to set the stage for our derivations. In particular, Anderson acceleration is illustrated in section~\ref{Anderson acceleration}
whereas in section~\ref{sec:boostconv} we
briefly recall the {\tt BoostConv} algorithm and its implementation. Section~\ref{sec:broyden} sees the main contribution of this paper. We start by
discussing the interpretation of {\tt BoostConv} in terms of Anderson acceleration. This will allow us to show the convergence properties of {\tt BoostConv}, provided a certain matrix computed by this scheme is full rank. Therefore, in section~\ref{Stabilized} we design a variant of {\tt BoostConv} that guarantees such property. The {\tt BoostConv} convergence results are then stated in section~\ref{Convergence properties}. In section~\ref{Numerical Results} we report a panel of different results to numerically validate our findings. The paper ends with section~\ref{Conclusions} where some conclusions are drawn.

%Given a generic nonlinear operator $\bF:\mathbb{R}^n\rightarrow\mathbb{R}^n$ we are interested in solving the problem
%
%\begin{equation}\label{eq:main}
%    \bF(x)=0.
%\end{equation}
%By linearizing around a current approximate solution $x_k$, $k\geq 0$, of~\eqref{eq:main}, for a given initial guess $x_0$, we get the linear system
%%
%\begin{equation}\label{eq:linearization}
%    \frac{\partial\bF}{\partial x}(x_{k+1}-x_k)=-\bF(x_k).
%\end{equation}
%Since the exact inversion of the Jacobian $\partial\bF/\partial x$ is seldom an option, we %rewrite~\eqref{eq:linearization} as 
%%
%\begin{equation}\label{eq:iteration1}
%        x_{k+1}=x_k+Br_k,
%\end{equation}
%%
%where $r_k=-\bF(x_k)$ and $B$ is a matrix such that $B\approx (\partial\bF/\partial x)^{-1}$\footnote{Notice that if $B= (\partial\bF/\partial x)^{-1}$ then~\eqref{eq:iteration1} amounts to the Newton method.}.

%%%%%%%%%%%%%%%%%%%%%%%%%%%%%%%%%%%%%%%%%%%%%%%%%%
\section{Background material}\label{Background material}
In this section we revise some preliminary material which serves as foundation for our novel derivations that will be presented in section~\ref{sec:broyden}.

%%%%%%%%%%%%%%%%%%%%%%%%%%%%%%%%%%%%%%%%%%%%%%%%%
\subsection{Anderson acceleration}\label{Anderson acceleration}
Consider the nonlinear problem~\eqref{eq:nonlinear_system} and assume that a certain number $k$ of steps of one's favorite iterative solver, e.g.,~\eqref{eq:basic_iteration}, has been performed. Anderson acceleration (or Anderson mixing) aims at determining a good estimate $x_{k+1}$ that approximates the solution to~\eqref{eq:nonlinear_system} without additional evaluations of $\bF$ but solely relying on the information that has been already constructed, namely $x_{k-j}$ and $\bF(x_{k-j})$ for $j=0,\ldots,m$, $0<m\le k$. 
% giusto?

By defining
$$\Delta X_k=[x_{k-m+1}-x_{k-m},\ldots,x_{k}-x_{k-1}],\qquad\Delta F_k=[\bF(x_{k-m+1})-\bF(x_{k-m}),\ldots,\bF(x_{k})-\bF(x_{k-1})],$$
Anderson acceleration computes $x_{k+1}$ as
$$
x_{k+1}=x_k+\beta \bF(x_k)-(\Delta X_k+\beta \Delta F_k)\gamma_k,
$$
where $\beta\in\mathbb{R}$ is called \emph{mixing parameter}\footnote{In its full generality, Anderson acceleration sees the employment of a mixing parameter that changes at each iteration so that using a $\beta_k$, in place of $\beta$, would be more appropriate. However, a constant $\beta_k\equiv\beta$ is often adopted in practice.} and $\gamma_k\in\mathbb{R}^m$ is such that 
$$
\gamma_k=\argmin_{\gamma\in\mathbb{R}^m}\|\bF(x_k)-\Delta F_k\gamma\|_2^2;
$$
see, e.g.,~\cite{FangSaas2009}.
Assuming $\Delta F_k$ to be full rank, we can thus rewrite Anderson acceleration as
\begin{equation}\label{eq:anderson_it}
x_{k+1}=x_k+\beta \bF(x_k)-(\Delta X_k+\beta \Delta F_k)\gamma_k=x_k+\beta \bF(x_k)-(\Delta X_k+\beta \Delta F_k)\Delta F_k^\dagger \bF(x_k)\rev{,}
\end{equation}
\rev{where $\Delta F_k^\dagger$ denotes the Moore--Penrose pseudo-inverse of $\Delta F_k$.}

Many papers available in the literature have further studied Anderson acceleration~\cite{WalkerNi2011}, its convergence~\cite{TothKelley2015,EvansEtAl2020}, and more advanced variants like a stabilized and/or periodic use of Anderson acceleration~\cite{brezinski_anderson}. Periodic Anderson acceleration employs Anderson acceleration every other $p>0$ ``plain" iterations of the form~\eqref{eq:basic_iteration}. Stabilized Anderson acceleration detects a possible numerical linear dependency among the columns of $\Delta F_k$ and retain\rev{s} only the most significant directions encoded in this matrix. In~\cite[Algorithm 6]{brezinski_anderson} this is done by employing a Gram-Schmidt approach where a column of $\Delta F_k$ is discarded if its orthogonal projection onto the subspace orthogonal to the one spanned by the other columns has a too small norm. If $\mathcal{I}_k$ denotes the index set related to columns of $\Delta F_k$ that are sufficiently linearly independent, and $\Delta F_k^{\mathcal{I}_k}$ denotes the submatrix of $\Delta F_k$ containing those columns, then the stabilized Anderson acceleration reads as follows
$$
x_{k+1}=x_k+\beta \bF(x_k)-(\Delta X_k^{\mathcal{I}_k}+\beta \Delta F_k^{\mathcal{I}_k})\rev{(\Delta F_k^{\mathcal{I}_k})^\dagger} \bF(x_k),
$$
where $\Delta X_k^{\mathcal{I}_k}$ contains the columns of $\Delta X_k$ whose index belongs to $\mathcal{I}_k$.

%%%%%%%%%%%%%%%%%%%%%%%%%%%%%%%%%%%%%%%%%%%%%%%%%%
\subsection{BoostConv}\label{sec:boostconv}
The main goal of {\tt BoostConv}~\cite{CITRO2017234} is to improve the convergence of \eqref{eq:basic_iteration}. This is achieved by replacing~\eqref{eq:basic_iteration} with 
\begin{equation}\label{eq:iteration2}
        x_{k+1}=x_k+B\xi_k,
\end{equation}
where $\xi_k=\xi_k(r_k)$ is a suitable function of the current residual. \rev{At $k=0$ no residual history is available and the first step is left unaccelerated, namely we set $\xi_0=r_0$, so that $x_1=x_0+Br_0$. For $k\geq1$, given} a parameter $N>0$, we define two matrices $W_k,V_k\in\mathbb{R}^{n\times \rev{m_k}}$, \rev{$m_k:=\min\{N,k\}$}, such that  

%%% questo suitable nasconde tutta la mia ignoranza in merito alle proprieta' di questa funzione :) magari si puo' dire qualcosa in merito

\begin{equation}\label{eq:def_W_V}
    W_k=\left\{\begin{array}{l}
    \xi_{k-1}+r_{k}-r_{k-1},\quad\text{if }N=1\text{ or } k=1,\\
    
        [\widetilde W_{k-1}, \xi_{k-1}+r_{k}-r_{k-1}], \quad \text{otherwise,} 
    \end{array}\right. \quad
V_k=\left\{\begin{array}{l}
    r_{k-1}-r_{k},\quad\text{if }N=1\text{ or } k=1,\\
    
        [\widetilde V_{k-1}, r_{k-1}-r_{k}], \quad \text{otherwise,} 
    \end{array}\right.
\end{equation}
In~\eqref{eq:def_W_V}, $\widetilde W_{k-1}$ and $\widetilde V_{k-1}$ are matrices with \rev{$\min\{N-1,k-1\}$} columns defined as follows
$$\widetilde W_{k-1}=\left\{\begin{array}{l}
     W_{k-1},\quad \text{if } 1<k\rev{\leq} N, \\
     
     [W_{k-1}e_2,\ldots,W_{k-1}e_N], 
     \quad \text{otherwise,} 
     \end{array}\right.\quad \widetilde V_{k-1}=\left\{\begin{array}{l}
     V_{k-1},\quad \text{if } 1<k\rev{\leq} N, \\
     
     [V_{k-1}e_2,\ldots,V_{k-1}e_N], 
     \quad \text{otherwise,} 
     \end{array}\right.
$$
which means that, at iteration $k>N$, we retain only the last $N-1$ columns of $W_{k-1}$ and $V_{k-1}$ and employ the newly available data to define the $N$th column of $W_k$ and $V_k$. \rev{Therefore, $W_k$ and $V_k$ have $k$ columns for $k\leq N$ and $N$ columns for $k>N$, i.e., $m_k=\min\{N,k\}$ columns.} 
\begin{remark}
It is important to realize that at iteration $k$, once $x_{k+1}$ is computed, the vectors $r_k$ and $\xi_k$ cannot be discarded as they are both needed to define the new column of $V_{k+1}$ and $W_{k+1}$, at the next iteration.
\end{remark}

With the matrices $W_k$ and $V_k$ at hand, we compute the vector $\xi_k$ in~\eqref{eq:iteration2} as 
\begin{equation}\label{eq:xi_def}
    \xi_k=r_k+W_kc_k,
\end{equation}
where $c_k\in\mathbb{R}^{\rev{m_k}}$ is such that 
\begin{equation}\label{eq:selection_ck}
    c_k\rev{\in}\argmin_{c\in\mathbb{R}^{\rev{m_k}}}\|r_k-V_kc\|_2^2\rev{,}
\end{equation}
\rev{where, whenever the least squares problem above admits multiple minimizers, i.e., $V_k$ is rank deficient, we select the one of minimum Euclidean norm, namely $c_k=V_k^\dagger r_k$, with $V_k^\dagger$ the Moore--Penrose pseudo-inverse of $V_k$. We will see in section~\ref{Stabilized} that, for the robust variant of {\tt BoostConv}, $V_k$ always has full column rank, so that $c_k$ is uniquely defined.}

Notice that the update~\eqref{eq:iteration2} does not need to be applied at every iteration of the underlying scheme~\eqref{eq:basic_iteration}.
In practice, one may instead activate {\tt BoostConv} every $p>0$ iterations, depending on the rate at which 
the local Jacobian of the nonlinear problem varies along the iteration. When the Jacobian changes slowly, 
sampling the residual recombination step less frequently may be sufficient and computationally convenient; 
conversely, when significant variations are observed, applying~\eqref{eq:iteration2} at each iteration can provide 
improved robustness and faster convergence. This choice does not alter the structure of the method, and the 
matrices $W_k$ and $V_k$ are defined as before, being updated only at the selected \textit{active} iterations, 
namely for $k$ such that $\mathrm{mod}(k,p)=0$.

%% forse conviene mettere un commento del perche' potrebbe essere necessario fare una call ogni tot iterazioni

%%%%%%%%%%%%%%%%%%%%%%%%%%%%%%%%%%%%%%%%%%%%
\section{{\tt BoostConv} as Anderson-like acceleration}\label{sec:broyden}
For the sake of simplicity, we assume that {\tt BoostConv} is applied at each iteration $k>0$ with a window $N\geq1$, namely, equation~\eqref{eq:basic_iteration} is replaced by~\eqref{eq:iteration2} at each $k$. %We then generalize our analysis to the case where {\tt BoostConv} is adopted only every other $p$ iterations.

Thanks to the selection of $c_k$ in~\eqref{eq:selection_ck}, the definition of $\xi_k$ in~\eqref{eq:xi_def} can be recast as 
$$\xi_k=r_k+W_kV_k^\dagger r_k,$$
where $V_k^\dagger$ denotes the Moore-Penrose pseudo-inverse of $V_k$, so that~\eqref{eq:iteration2} becomes
\begin{equation}\label{eq:fixedpoint_boostconv1}
    x_{k+1}=x_k+B(I+W_kV_k^\dagger )r_k=x_k-B(I+W_kV_k^\dagger )\bF(x_k).
\end{equation}
Now, we are going to have a closer look at the matrices $W_k$ and $V_k$.  If $k> N$, so that $W_k$ and $V_k$ have both $N$ columns, 
by noticing that $\xi_k-r_k=W_kV_k^\dagger r_k$ for any $k\rev{\geq1}$,
we have\footnote{If $k\leq N$, $W_k$ and $V_k$ have only \rev{$k$} columns. Nevertheless, for $k\geq 1$, all these columns have the form depicted in~\eqref{eq:def2W}--\eqref{eq:def2V}\rev{, where $\xi_0-r_0=0$ for the first column}.}
\begin{align}\label{eq:def2W}
W_k=&\,[\xi_{k-N}-r_{k-N}+r_{k-N+1},\ldots,\xi_{k-1}-r_{k-1}+r_{k}]\notag\\
=&\, [W_{k-N}V_{k-N}^\dagger r_{k-N}+r_{k-N+1}, \ldots, W_{k-1}V_{k-1}^\dagger r_{k-1}+r_{k}]\notag\\
=&-[W_{k-N}V_{k-N}^\dagger \bF(x_{k-N})+\bF(x_{k-N+1}), \ldots, W_{k-1}V_{k-1}^\dagger \bF(x_{k-1})+\bF(x_{k})].
\end{align}
and
\begin{equation}\label{eq:def2V}
    V_k=[r_{k-N}-r_{k-N+1}, \ldots, r_{k-1}-r_{k}]=[\bF(x_{k-N+1})-\bF(x_{k-N}), \ldots, \bF(x_{k})-\bF(x_{k-1})].
\end{equation}
By relying on~\eqref{eq:fixedpoint_boostconv1}, a direct computation shows that
$$-B W_kV_k^\dagger \bF(x_k)=x_{k+1}-x_k+B \bF(x_k).$$
This means that we can rewrite $W_k$ as
\begin{align*}
BW_k=&\, [x_{k-N+1}-x_{k-N},\ldots,x_k-x_{k-1}]-B[\bF(x_{k-N+1})-\bF(x_{k-N}),\ldots, \rev{\bF(x_{k})-\bF(x_{k-1})}]\\
=&\,  [x_{k-N+1}-x_{k-N},\ldots,x_k-x_{k-1}]-BV_k.
\end{align*}
By denoting $\Delta X_k:=[x_{k-N+1}-x_{k-N},\ldots,x_k-x_{k-1}]$ and $\Delta F_k=V_k$, we can recast~\eqref{eq:fixedpoint_boostconv1} as
\begin{equation}\label{eq:fixedpoint_boostconv2}
 x_{k+1}=x_k-B \bF(x_k)- (\Delta X_k-B \Delta F_k)\Delta F_k^\dagger \bF(x_k).
\end{equation}
The recursion in~\eqref{eq:fixedpoint_boostconv2}
can be seen as the update provided by the generalized Broyden's method of the second kind; cf., e.g.,~\cite[Equation (17)]{FangSaas2009} for $G_{k-m}=B$. Equation~\eqref{eq:fixedpoint_boostconv2} shows that, at each iteration $k$, {\tt BoostConv} employs an approximation to the inverse of the Jacobian $\partial \bF/\partial x$ given by 
$B+(\Delta X_k-B \Delta F_k)\Delta F_k^\dagger$ in place of \rev{$B$ alone}. The low-rank update $(\Delta X_k-B \Delta F_k)\Delta F_k^\dagger$ is crucial for the success of the overall procedure. In particular, it can be shown that, for a full-rank $\Delta F_k$, $B+(\Delta X_k-B \Delta F_k)\Delta F_k^\dagger$ is the only matrix that satisfies both the (multi)secant condition 
$$G\Delta F_k=\Delta X_k,$$
(in the unknown $G$)
and the so-called \emph{no-change} condition
$$(G-B)q=0, \quad \text{for all }q\perp\text{Range}(\Delta F_k).$$
Moreover,
$$B+(\Delta X_k-B \Delta F_k)\Delta F_k^\dagger=\argmin_{G\Delta F_k=\Delta X_k}\|G-B\|_F,$$
and it is worth noticing that,
if $\Delta F_k$ is \rev{square} and full rank, then~\eqref{eq:fixedpoint_boostconv2} reads as
$$x_{k+1}=x_k-\Delta X_k\Delta F_k^{-1} \bF(x_k), $$
for any $B$, that is {\tt BoostConv} is indeed a (multi)secant method.

\begin{remark}\label{remark_broyden}
Even though~\eqref{eq:iteration2} and~\eqref{eq:fixedpoint_boostconv2} are mathematically equivalent, it is interesting to notice that the former leaves the ``iteration'' matrix $B$ untouched but it applies it to a different vector: $\xi_k$ in place of $r_k$. The formulation~\eqref{eq:fixedpoint_boostconv2}, instead, keeps $\bF(x_k)$ untouched but modifies the iteration matrix that now is of the form $B+(\Delta X_k-B \Delta F_k)\Delta F_k^\dagger$ in place of $B$. The less intrusive nature of the {\tt BoostConv} formulation~\eqref{eq:iteration2} is very advantageous when this strategy needs to be merged in pre-existing code, in a black-box manner. Indeed, once $r_k$ is available, one just computes $\xi_k$ and runs~\eqref{eq:iteration2} instead of~\eqref{eq:basic_iteration}. On the other hand, classic implementations of the Broyden method see an actual update to the current approximation to the inverse Jacobian of the form 
$B_{k+1}=B_k+(\Delta X_k-B_k \Delta F_k)\Delta F_k^\dagger$. This difference will not allow for the convergence properties of the Broyden method to straightforwardly carry over to {\tt BoostConv} as well.
\end{remark}

Clearly, if $B=-\beta I$, $\beta>0$ in~\eqref{eq:iteration2}, {\tt BoostConv} \rev{coincides with}
Anderson acceleration; cf., e.g.,~\cite[Equation (24)]{FangSaas2009} and equation~\eqref{eq:anderson_it}.
 For this specific choice of $B$, \rev{the convergence properties of Anderson acceleration directly carry over to the {\tt BoostConv}
implementation.}
For a more generic $B$, \rev{one may be tempted to} rewrite the basic {\tt BoostConv} iteration~\eqref{eq:iteration2} as 
$$x_{k+1}=x_k+\beta(1/\beta B)r_k=x_k+\beta \widetilde r_k,\quad \widetilde r_k=-1/\beta B\bF(x_k),$$
namely \rev{as} a linear update scheme \rev{applied} to the \emph{preconditioned} problem $1/\beta B\bF(x)=0$\rev{, and to invoke the convergence theory of Anderson acceleration for the latter problem. However, Anderson acceleration applied to the preconditioned problem computes its coefficients by solving $\min_{\gamma}\|B(\bF(x_k)-\Delta F_k\gamma)\|_2$, whereas {\tt BoostConv} solves $\min_{\gamma}\|\bF(x_k)-\Delta F_k\gamma\|_2$; cf.~\eqref{eq:selection_ck} and~\eqref{eq:def2V}. These two least squares problems share the same minimizer when $B^TB$ is a positive multiple of the identity, namely when $B$ is a nonzero multiple of an orthogonal matrix, but in general they lead to different coefficients and thus to different iterates. Moreover, if $B$ is singular, the preconditioned problem is not equivalent to~\eqref{eq:nonlinear_system}. Therefore, the convergence theory of Anderson acceleration does not directly transfer to {\tt BoostConv} with a general $B$.} We thus \rev{pursue} a different path by interpreting $B$ as a \emph{multi-mixing} parameter and generaliz\rev{ing} some of the convergence results of Anderson acceleration ($B=-\beta I$) to this more generic setting; see section~\ref{Convergence properties}. 

%%%%%%%%%%%%%%%%%%%%%%%%%%%%%%%%%%%%%%%
\subsection{Robust {\tt BoostConv}}\label{Stabilized}
To show convergence results for {\tt BoostConv} we will exploit the relation between the latter and Anderson acceleration. In particular, thanks to their mild assumptions, we would like to adapt the results in~\cite[Section 5]{brezinski_anderson} to the case of {\tt BoostConv}. To this end, we first need to design a \emph{robust} version of {\tt BoostConv}, and this is the goal of this section. Indeed,
in~\cite[Theorem 5.8]{brezinski_anderson} the authors show linear convergence of stabilized Anderson acceleration and not of the plain, regular scheme.

As mentioned in section~\ref{Anderson acceleration}, stabilized Anderson acceleration detects possible linear dependency 
among the columns of $\Delta F_k$ by employing a Gram-Schmidt approach. For {\tt BoostConv}, this translates into the same steps but for the matrix $V_k$ in~\eqref{eq:def_W_V}. The Gram-Schmidt approach in~\cite[Algorithm 6]{brezinski_anderson} can be seen as an updating of the skinny QR factorization of $V_k$, $V_k=Q_kR_k$, from one iteration to the next one, equipped with a check on the possible linear dependency of the newly acquired columns. As a byproduct, this update provides us also with a cheap procedure for solving the least squares problem~\eqref{eq:selection_ck}. \rev{A closely related filtering strategy for Anderson acceleration, which discards the columns of $\Delta F_k$ forming a too small angle with the subspace spanned by the retained ones, has been recently proposed and analyzed in~\cite{PollockRebholz2023}; see also~\cite{PollockRebholzBook}.}

We define the matrices $Q_k\in\mathbb{R}^{n\times \widehat N\rev{_k}}$ and $R_k\in\mathbb{R}^{\widehat N\rev{_k}\times \widehat N\rev{_k}}$, \rev{where $0\leq\widehat N_k\leq m_k=\min\{N,k\}$ denotes the number of columns actually retained at iteration $k$}, as follows. While nothing happens for $k=0$ \rev{($\xi_0=r_0$)}, starting from $k=1$, we have
$$Q_1=(r_0-r_1)/\|r_0-r_1\|,\quad R_1=\|r_0-r_1\|,$$
so that $c_1=R_1^{-1}(Q_1^Tr_1)$ in~\eqref{eq:selection_ck}\rev{, provided that $r_0\neq r_1$. Otherwise, no column is retained, $\widehat N_1=0$, and we simply set $\xi_1=r_1$. More generally, whenever the retained history is empty we set $\xi_k=r_k$. We also recall that the outer iteration~\eqref{eq:iteration2} is stopped as soon as $\|r_k\|$ falls below a prescribed tolerance, so that we can assume $r_k\neq0$ throughout.}

For $k>1$, if $\widehat N\rev{_{k-1}}< N$, we just need the new column $r_{k-1}-r_k$ of $V_k=[V_{k-1},r_{k-1}-r_k]$ to get orthogonalized with respect to the columns of $Q_{k-1}$. In particular, we compute
$$\widetilde q_k=(I-Q_{k-1}Q_{k-1}^T)(r_{k-1}-r_k).$$
If $\|\widetilde q_k\|\rev{\leq}\tau\|r_{k-1}-r_k\|$ for a prescribed threshold $\tau\rev{\in(0,1)}$ \rev{(this includes the case $r_{k-1}=r_k$)}, then $r_{k-1}-r_k$ get\rev{s} discarded and we define $V_k=V_{k-1}$, $W_k=W_{k-1}$, $Q_k=Q_{k-1}$, $R_k=R_{k-1}$. Otherwise, we set $q_k=\widetilde q_k/\|\widetilde q_k\|$ and 
$$Q_k=[Q_{k-1},q_k],\qquad R_k=\begin{bmatrix}
R_{k-1} & Q_{k-1}^T(r_{k-1}-r_k)\\
 & \|\widetilde q_k\|\\
\end{bmatrix}.
$$
In either case, we can compute $c_k=R_k^{-1}(Q_k^Tr_k)$.

If\rev{, at} a certain iteration $k$, $V_{k-1}$ already has $N$ columns \rev{($\widehat N_{k-1}=N$)}, the computation of the QR factors of $V_k$ requires first a downdating. Indeed, $V_k$ is obtained by removing the first column of $V_{k-1}$ and then adding a new column; cf., the definition of $\widetilde V_{k-1}$ in section~\ref{sec:boostconv}. If $V_{k-1}=Q_{k-1}R_{k-1}$ is such that 
$$ \rev{R_{k-1}}=\begin{bmatrix}
\alpha & \ell^T \\
 & \widehat R_{k-1}\\
\end{bmatrix},
$$
where $\alpha\in\mathbb{R}$, $\ell\in\mathbb{R}^{N-1}$, and $\widehat R_{k-1}\in\mathbb{R}^{(N-1)\times (N-1)}$, then it is easy to show that a skinny QR decomposition of $\widetilde V_{k-1}=V_{k-1}[e_2,\ldots,e_N]$ can be cheaply computed; see, e.g.,~\cite[Section 12.5.2]{Golub2013}. In particular, we have
$$\widetilde V_{k-1}=V_{k-1}[e_2,\ldots,e_N]=Q_{k-1}R_{k-1}[e_2,\ldots,e_N]=Q_{k-1}\begin{bmatrix}
\ell^T \\
 \widehat R_{k-1}\\
\end{bmatrix},$$
and since the right-most matrix in the relation above is upper Hessenberg, it can be triangularized by applying, e.g., $N-1$ Givens rotation\DP{s}. Let $P\in\mathbb{R}^{N\times N}$ be the orthogonal matrix collecting these rotations, namely $P$ is such that 
$$
P\begin{bmatrix}
\ell^T \\
 \widehat R_{k-1}\\
\end{bmatrix}=\begin{bmatrix}
\widetilde R_{k-1} \\
 0\\
\end{bmatrix}.
$$
Then, a skinny QR of $\widetilde V_{k-1}$ is given by $\widetilde V_{k-1}=\widetilde Q_{k-1}\widetilde R_{k-1}$ where $\widetilde Q_{k-1}=Q_{k-1}P^T[e_1,\ldots,e_{N-1}]$. Once this is done, the skinny QR of $V_k=[\widetilde V_{k-1},r_{k-1}-r_k]=Q_kR_k$ is computed by updating $\widetilde Q_{k-1}$ and $\widetilde R_{k-1}$ as above. In Algorithm~\ref{alg:stabilizedBoostConv} we report this robust variant of the {\tt BoostConv} procedure. Notice that in this case we no longer need to store the matrix $V_k$ but we rather update its QR factors. \rev{Notice also that, thanks to the convention that $W_0$, $Q_0$, and $R_0$ are empty matrices, the case $k=1$ does not require a separate treatment in Algorithm~\ref{alg:stabilizedBoostConv}.}  

 \begin{algorithm}[t]
\begin{algorithmic}[1]
%\setstretch{1.2}
\smallskip
\Statex \textbf{Input:} Matrices $W_{k-1}$, $Q_{k-1}\in\mathbb{R}^{n\times \widehat N_{k-1}}$, $R_{k-1}\in\mathbb{R}^{\widehat N_{k-1}\times \widehat N_{k-1}}$, \rev{where $0\leq \widehat N_{k-1}\leq\min\{N,k-1\}$ is the number of retained columns (empty matrices if $\widehat N_{k-1}=0$, e.g., for $k=1$)}, residual vectors $r_k$, $r_{k-1}$, and $\xi_{k-1}$, threshold $\tau\rev{\in(0,1)}$.
\Statex \textbf{Output:} Vector $\xi_k$ and updated matrices $W_{k}$, $Q_{k}$, and $R_{k}$.
\smallskip
\If{\rev{$\widehat N_{k-1}=N$}}
\State Set $\widetilde W_{k-1}=W_{k-1}[e_2,\ldots,e_N]$
\State Compute $P\in\mathbb{R}^{N\times N}$ s.t. $PR_{k-1}[e_2,\ldots,e_N]=\begin{bmatrix}
\widetilde R_{k-1} \\
 0\\
\end{bmatrix}$ \label{alg_line_updateR}
\State Set $\widetilde Q_{k-1}=Q_{k-1}P^T[e_1,\ldots,e_{N-1}]$ \label{alg_line_updateQ}
\Else
\State $\widetilde W_{k-1}=W_{k-1}$, $\widetilde Q_{k-1}=Q_{k-1}$, $\widetilde R_{k-1}=R_{k-1}$

\EndIf
\State Compute $\widetilde q_k=(I-\widetilde Q_{k-1}\widetilde Q_{k-1}^T)(r_{k-1}-r_k)$ \label{alg_line_GS}
\If{$\|\widetilde q_k\|\rev{\leq}\tau\|r_{k-1}-r_k\|$}
\State Set $W_k=W_{k-1}$, $Q_k=Q_{k-1}$, $R_k=R_{k-1}$ \rev{(the new column is discarded)}
\Else 
\State $W_k=[\widetilde W_{k-1},\xi_{k-1}+r_{k}-r_{k-1}]$
\State $Q_k=[\widetilde Q_{k-1},\widetilde q_k/\|\widetilde q_k\|]$
\State $R_k=\begin{bmatrix}
\widetilde R_{k-1} & \widetilde Q_{k-1}^T(r_{k-1}-r_k)\\
 & \|\widetilde q_k\|\\
\end{bmatrix}$

\EndIf

\State Set $\xi_k=r_k+W_kR_k^{-1}Q_k^Tr_k$ \rev{($\xi_k=r_k$ if $W_k$ is empty)} 
\end{algorithmic}    \caption{Robust {\tt BoostConv}.  \label{alg:stabilizedBoostConv} }
\end{algorithm}

\begin{remark}
    We would like to stress that \rev{the rejection test in Algorithm~\ref{alg:stabilizedBoostConv} comes at no extra cost}. In particular, the detection of a possible linear dependency between columns of $V_k$ comes for free when updating the QR factorization of the latter matrix\rev{, so that the robust {\tt BoostConv} scheme has essentially the same cost per iteration as a plain {\tt BoostConv} implementation that updates the QR factors of $V_k$ along the iterations, namely $\mathcal{O}(nN+N^2)$ floating point operations (flops):} $\mathcal{O}(nN)$ flops for the Gram-Schmidt step (line~\ref{alg_line_GS}) and the update of $\widetilde Q_{k-1}$ (line~\ref{alg_line_updateQ}), and other $\mathcal{O}(N^2)$ flops for the update of $\widetilde R_{k-1}$ (line~\ref{alg_line_updateR}). \rev{On the other hand, the plain {\tt BoostConv} implementation of~\cite{CITRO2017234} solves the least squares problem~\eqref{eq:selection_ck} from scratch at each iteration, e.g., by computing the skinny QR factorization\footnote{Or different decompositions but with a similar computational cost.} of $V_k$ at a cost of $\mathcal{O}(nN^2)$ flops, and it does not detect possible (numerical) rank deficiencies of $V_k$.} 
\end{remark}

\begin{remark}\label{remark_stabilizedBoost}
    It is easy to show that the employment of the robust {\tt BoostConv} implementation leads to an update of the current solution of the form 
    \begin{equation}\label{eq:iteration_boostconv_stabilized}
            x_{k+1}=x_k+B\xi_k=x_k-B \bF(x_k)- (\Delta X_k^{\mathcal I_k}-B \Delta F_k^{\mathcal I_k})(\Delta F_k^{\mathcal I_k})^\dagger \bF(x_k),
                \end{equation}
    where $\Delta X_k^{\mathcal I_k}=\Delta X_{k-1}\rev{^{\mathcal I_{k-1}}}$ if Algorithm~\ref{alg:stabilizedBoostConv} detects linear dependency, \rev{whereas $\Delta X_k^{\mathcal I_k}$ is obtained by appending the column $x_k-x_{k-1}$ to $\Delta X_{k-1}^{\mathcal I_{k-1}}$ (after removing its first column if $\widehat N_{k-1}=N$)} otherwise. The same for $\Delta F_k^{\mathcal I_k}$. This means that also the update provided by Algorithm~\ref{alg:stabilizedBoostConv} satisfies conditions similar to the ones stated in section~\ref{sec:broyden} for plain {\tt BoostConv}. In particular, the matrix $B + (\Delta X_k^{\mathcal I_k}-B \Delta F_k^{\mathcal I_k})(\Delta F_k^{\mathcal I_k})^\dagger$ is such that 
    $$G\Delta F_k^{\mathcal I_k}=\Delta X_k^{\mathcal I_k},$$
and 
$$(G-B)q=0, \quad \text{for all }q\perp\text{Range}(\Delta F_k^{\mathcal I_k}).$$
Moreover,
$$B+(\Delta X_k^{\mathcal I_k}-B \Delta F_k^{\mathcal I_k})(\Delta F_k^{\mathcal I_k})^\dagger=\argmin_{G\Delta F_k^{\mathcal I_k}=\Delta X_k^{\mathcal I_k}}\|G-B\|_F.$$

\end{remark}

%%%%%%%%%%%%%%%%%%%%%%%%%%%%%%%%%%%%%%%
\subsection{Convergence properties}\label{Convergence properties}

In this section we adapt the convergence results 
of stabilized Anderson acceleration presented in~\cite[Section 5]{brezinski_anderson} to the case of robust {\tt BoostConv}. We first state the following assumption; cf., ~\cite[Assumption 5.2]{brezinski_anderson}.
\begin{assumption}\label{Assumption}
Let $\bF:\mathbb{R}^n\rightarrow\mathbb{R}^n$ in~\eqref{eq:nonlinear_system} be differentiable in \rev{an} open convex set $E\subset\mathbb{R}^n$ and assume there exists $x_*\in E$
such that $\bF(x_*)=0$. Moreover, assume $\partial \bF/\partial x (x_*)$
to be invertible and that for all $\widehat x\in E$ it holds
$$\|\partial \bF/\partial x (\widehat x)-\partial \bF/\partial x (x_*)\|\leq L\|\widehat x-x_*\|,$$
for $L\geq0$.
\end{assumption}
Notice that the assumption above implies that 
$$\|\bF(y)-\bF(z) -\partial \bF/\partial x (y-z)\|\leq L\|y-z\|\cdot\max\{\|y-x_*\|,\|z-x_*\|\},$$
for all $y,z\in E$ and that there exists a neighborhood of $x_*$, that we denote by $U_\kappa(x_*):=\{y\in\mathbb{R}^n \text{ s.t. }\|y-x_*\|\leq \kappa\}$\rev{, with $U_\kappa(x_*)\subseteq E$,} such that, for some $\rho>0$, it holds
$$\frac{1}{\rho}\|y-z\|\leq\|\bF(y)-\bF(z)\|\leq\rho\|y-z\|.$$

This assumption is standard in the local convergence analysis of Newton and quasi-Newton methods and is not restrictive in the present setting. In particular, the differentiability of $\bF$ together with the local Lipschitz continuity of its Jacobian are mild regularity requirements satisfied by a broad class of nonlinear problems arising from the \rev{discretization} of smooth differential equations and nonlinear residual operators. Furthermore, invertibility of the derivative of $ \bF$ with respect to $x$ in $x_\ast$ corresponds to the classical nonsingularity condition at a regular solution. This guarantees the local well-posedness of the problem and underlies most local convergence results for multisecant and Anderson-type acceleration schemes. Consequently, Assumption \ref{Assumption} aligns with the standard analytical framework adopted in the literature and does not impose additional structural constraints beyond those commonly required for local analysis.

%{\color{red} @Vincenzo: can you add some comments about how this assumption is not restrictive in our setting?}

For the sake of simplicity, we assume that Algorithm~\ref{alg:stabilizedBoostConv} is applied at each iteration $k$ in~\eqref{eq:basic_iteration}, namely all the outer iterations have the form~\eqref{eq:iteration2} with $\xi_k$ computed by robust {\tt BoostConv}. In this scenario, under the hypotheses of Assumption~\ref{Assumption}, if $B=-\beta I$ in~\eqref{eq:iteration2}, the results in~\cite[Theorem 5.8]{brezinski_anderson} hold for Algorithm~\ref{alg:stabilizedBoostConv}. However, we would like to derive convergence results for generic $B$. This requires some technical modifications to the contributions in~\cite[Section 5]{brezinski_anderson} that we elaborate in this section. We first recall the definition of the Frobenius norm of a matrix $X\in\mathbb{R}^{m\times n}$
$$\|X\|_F^2:=\text{trace}(X^TX)=\sum_{i=1}^n\|Xe_i\|^2.$$

\begin{lemma}\label{lemma5.6}
At iteration $k$ of the form~\eqref{eq:iteration_boostconv_stabilized}, if $\widehat N\rev{=\widehat N_k}$ denotes the cardinality of $\mathcal{I}_k$, let \rev{all the iterates involved in the definition of $\Delta X_k^{\mathcal{I}_k}$ and $\Delta F_k^{\mathcal{I}_k}$}
 belong to $U_\kappa(x_*)$. Then, it holds
$$\|\widehat{\Delta X}^{\mathcal{I}_k}_ke_j-(\partial \bF/\partial x)^{-1}q_j\|\leq C\kappa_F(\rev{\bar R_k})\sum_{p=1}^j\ell_p,\quad\text{for all }j=1,\ldots,\widehat N,$$
where $\widehat{\Delta X}_k^{\mathcal I_k}:={\Delta X}_k^{\mathcal I_k}R_k^{-1}$, $q_j$ denotes the $j$th column of the orthogonal matrix $Q_k=[q_1,\ldots, q_{\widehat N}]\in\mathbb{R}^{n\times \widehat N}$ provided by Algorithm~\ref{alg:stabilizedBoostConv}, 
$C=\|(\partial \bF/\partial x)^{-1}\|L\rho$, \rev{$\ell_p=\max\{\|y_p-x_*\|,\|z_p-x_*\|\}$, with $y_p$, $z_p$ the iterates such that $\Delta X_k^{\mathcal{I}_k}e_p=y_p-z_p$ and $\Delta F_k^{\mathcal{I}_k}e_p=\bF(y_p)-\bF(z_p)$, $\bar R_k:=R_kD_k^{-1}$, with $D_k:=\text{diag}(\|R_ke_1\|,\ldots,\|R_ke_{\widehat N}\|)$, so that $\bar R_k$ is the triangular factor of the skinny QR factorization $\Delta F_k^{\mathcal{I}_k}D_k^{-1}=Q_k\bar R_k$ of the matrix $\Delta F_k^{\mathcal{I}_k}$ with normalized columns,} and $\kappa_F(\rev{\bar R_k})$ denotes the condition number, in the Frobenius norm, of $\rev{\bar R_k}$, namely $\kappa_F(\rev{\bar R_k})=\|\rev{\bar R_k}\|_F\|\rev{\bar R_k^{-1}}\|_F$.
\end{lemma}

\begin{proof}
%For $j=1$, we have
%\begin{align*}
%\|\widehat{\Delta X}^{\mathcal{I}_k}_ke_1-(\partial \bF/\partial x)^{-1}q_1\|=&\,\|\Delta X^{\mathcal{I}_k}_kR_k^{-1}e_1-(\partial \bF/\partial x)^{-1}q_1\|\\
%=&\,\left\|-\frac{1}{(R_k)_{1,1}}(\partial \bF/\partial x)^{-1}\left(\Delta F^{\mathcal{I}_k}_ke_1-\partial \bF/\partial x\Delta X^{\mathcal{I}_k}_ke_1\right)\right\|\\
%\leq&\, L\frac{\|(\partial \bF/\partial x)^{-1}\|}{|(R_k)_{1,1}|} 
%\|\Delta X^{\mathcal{I}_k}_ke_1\|\ell_1\leq L\rho\frac{\|(\partial \bF/\partial x)^{-1}\|}{|(R_k)_{1,1}|} 
%\|\Delta F^{\mathcal{I}_k}_ke_1\|\ell_1\\
%=&\,C\ell_1,
%\end{align*}
%where the inequalities are due to Assumption~\ref{Assumption} whereas the last equality comes from the fact that $(R_k)_{1,1}=\|\Delta F^{\mathcal{I}_k}_ke_1\|$ by construction.

For a generic $j\in\{1,\ldots,\widehat N\}$, we have
\begin{align*}
\|\widehat{\Delta X}^{\mathcal{I}_k}_ke_j-(\partial \bF/\partial x)^{-1}q_j\|=&\,\|\left(\Delta X^{\mathcal{I}_k}_k-(\partial \bF/\partial x)^{-1} \Delta F^{\mathcal{I}_k}_k\right)R_k^{-1}e_j\|\\
=&\,\left\|-(\partial \bF/\partial x)^{-1}\left(\Delta F^{\mathcal{I}_k}_k-\partial \bF/\partial x\Delta X^{\mathcal{I}_k}_k\right)R_k^{-1}e_j\right\|\\
=&\,\left\|-(\partial \bF/\partial x)^{-1}\sum_{p=1}^j\left(\Delta F^{\mathcal{I}_k}_ke_p-\partial \bF/\partial x\Delta X^{\mathcal{I}_k}_ke_p\right)(R_k^{-1})_{p,j}\right\|\\
\leq&\, L\|(\partial \bF/\partial x)^{-1}\|
\sum_{p=1}^j\|\Delta X^{\mathcal{I}_k}_ke_p\|\ell_p|(R_k^{-1})_{p,j}|\\
\leq &\,L\rho\|(\partial \bF/\partial x)^{-1}\| 
\sum_{p=1}^j\|\Delta F^{\mathcal{I}_k}_ke_p\|\ell_p|(R_k^{-1})_{p,j}|\\ 
=&\, C \sum_{p=1}^j\|R_ke_p\||(R_k^{-1})_{p,j}|\ell_p\rev{=C \sum_{p=1}^j|(\bar R_k^{-1})_{p,j}|\ell_p},
\end{align*}
where in the third equality we exploited the upper triangular pattern of $R_k^{-1}$, in the first and second inequality Assumption~\ref{Assumption}, whereas for the \rev{last two equalities} we recall that  $\Delta F_k^{\mathcal{I}_k}=Q_kR_k$ \rev{and that $R_k^{-1}=D_k^{-1}\bar R_k^{-1}$, so that $(R_k^{-1})_{p,j}=(\bar R_k^{-1})_{p,j}/\|R_ke_p\|$}.
Moreover, since \rev{the columns of $\bar R_k$ have unit norm, it holds $\|\bar R_k\|_F=\sqrt{\widehat N}\geq1$, whereas $\|\bar R_k^{-1}\|_F\geq |(\bar R_k^{-1})_{p,j}|$ for any $p$ and $j$, so that $|(\bar R_k^{-1})_{p,j}|\leq\kappa_F(\bar R_k)$}. Therefore, we can write
$$\|\widehat{\Delta X}^{\mathcal{I}_k}_ke_j-(\partial \bF/\partial x)^{-1}q_j\|\leq C\kappa_F(\rev{\bar R_k}) \sum_{p=1}^j\ell_p,
$$
getting the result.
\end{proof}

We would like to highlight the role of the condition number $\kappa_F(\rev{\bar R_k})$. This encodes the linear dependency in the columns of $\Delta F_k^{\mathcal{I}_k}$ and it could become dramatically large in case of a (almost) rank deficient $\Delta F_k^{\mathcal{I}_k}$. \rev{The next result shows that this scenario is avoided thanks to the rejection test implemented in Algorithm~\ref{alg:stabilizedBoostConv}, which keeps $\kappa_F(\bar R_k)$ uniformly bounded. We stress that this is true for the scaled triangular factor $\bar R_k$, i.e., for the matrix $\Delta F_k^{\mathcal{I}_k}$ with normalized columns, and not for $R_k$ itself. Indeed, the test in Algorithm~\ref{alg:stabilizedBoostConv} controls the angle between a new column and the subspace spanned by the retained ones, but not the relative length of the columns; see~\cite{PollockRebholz2023} for a filtering strategy that controls both. For instance, $\Delta F_k^{\mathcal I_k}=\text{diag}(1,s)$, $0<s\ll1$, has orthogonal columns and passes the test for any $\tau\in(0,1)$, while $\kappa_F(R_k)=\sqrt{(1+s^2)(1+s^{-2})}\to\infty$ as $s\to0$. It is thus important that the bound in Lemma~\ref{lemma5.6} only involves $\kappa_F(\bar R_k)$, which is insensitive to the column scaling.}

\rev{
\begin{lemma}\label{lemma_kappa}
Let $\bar R_k$ be defined as in Lemma~\ref{lemma5.6} and let $\tau\in(0,1)$ be the threshold employed in Algorithm~\ref{alg:stabilizedBoostConv}. Then, for any $k$ such that $\widehat N_k\geq1$,
$$\kappa_F(\bar R_k)\leq\bar\kappa(\tau,N):=N^{3/2}\tau^{-1}(1+\tau^{-1})^{N-1}.$$
\end{lemma}
\begin{proof}
By construction, the columns of $\bar R_k$ have unit norm, so that $|(\bar R_k)_{p,j}|\leq1$ for all $p,j$, and $\|\bar R_k\|_F=\sqrt{\widehat N}\leq\sqrt N$. Moreover, $|(\bar R_k)_{j,j}|=|(R_k)_{j,j}|/\|R_ke_j\|$ is the ratio between the norm of the component of the $j$th column of $\Delta F_k^{\mathcal I_k}$ orthogonal to the subspace spanned by the previous $j-1$ retained columns and the norm of the column itself. When this column was accepted by the test in Algorithm~\ref{alg:stabilizedBoostConv}, this ratio was larger than $\tau$, and it can only increase when some of the previous columns are removed by the downdating step. Hence, $|(\bar R_k)_{j,j}|\geq\tau$ for all $j$ (for $j=1$ the ratio is equal to one). Now, write $\bar R_k=D(I+U)$ with $D=\text{diag}((\bar R_k)_{1,1},\ldots,(\bar R_k)_{\widehat N,\widehat N})$ and $U$ strictly upper triangular with $|U_{p,j}|\leq\tau^{-1}$. Since $U$ is nilpotent, $\bar R_k^{-1}=\sum_{i=0}^{\widehat N-1}(-U)^iD^{-1}$ and, denoting by $S$ the strictly upper triangular matrix of all ones, a componentwise bound gives $|(I+U)^{-1}|\leq\sum_{i=0}^{\widehat N-1}(\tau^{-1}S)^i=(I-\tau^{-1}S)^{-1}$, whose $(p,j)$ entry is $\tau^{-1}(1+\tau^{-1})^{j-p-1}$ for $j>p$. Therefore, since $\bar R_k^{-1}=(I+U)^{-1}D^{-1}$, \DP{and $|D_{j,j}|\geq\tau$ so that $1/|D_{j,j}|\leq\tau^{-1}\leq1+\tau^{-1}$}, we get $|(\bar R_k^{-1})_{p,j}|\leq\tau^{-1}(1+\tau^{-1})^{j-p}$ for all $j\geq p$. \DP{This means} that $\|\bar R_k^{-1}\|_F\leq\widehat N\tau^{-1}(1+\tau^{-1})^{\widehat N-1}\leq N\tau^{-1}(1+\tau^{-1})^{N-1}$, and the result follows.
\end{proof}
}

\begin{lemma}\label{lemma5.7}
At iteration $k$ of the form~\eqref{eq:iteration_boostconv_stabilized}, the following equality hold\rev{s}
\begin{align*}
B + (\Delta X_k^{\mathcal I_k}-B \Delta F_k^{\mathcal I_k})(\Delta F_k^{\mathcal I_k})^\dagger-(\partial\bF/\partial x)^{-1}=&\,(B-(\partial\bF/\partial x)^{-1})(I-Q_kQ_k^T)\\
&\,+(\widehat{\Delta X}_k^{\mathcal I_k}-(\partial\bF/\partial \rev{x})^{-1}Q_k)Q_k^T.
\end{align*}
Moreover, under the hypotheses in Lemma~\ref{lemma5.6}, further assuming that $\ell_p\leq \varepsilon$ for all $p=1,\ldots,\widehat N$, then \rev{it holds}
$$\|(\widehat{\Delta X}_k^{\mathcal I_k}-(\partial\bF/\partial \rev{x})^{-1}Q_k)Q_k^T\|_F\leq C \kappa_F(\rev{\bar R_k})\varepsilon\alpha\rev{(\widehat N)},\qquad \rev{\alpha(\widehat N):=\sqrt{\frac{\widehat N(\widehat N+1)(2\widehat N+1)}{6}}\leq\alpha(N).}$$
\end{lemma}

\begin{proof}
The first equality comes from a direct computation recalling that 
$$B + (\Delta X_k^{\mathcal I_k}-B \Delta F_k^{\mathcal I_k})(\Delta F_k^{\mathcal I_k})^\dagger=B + (\widehat{\Delta X}_k^{\mathcal I_k}-B Q_k)Q_k^T.$$
Moreover, thanks to Lemma~\ref{lemma5.6}, we have
\begin{align*}
\|(\widehat{\Delta X}_k^{\mathcal I_k}-(\partial\bF/\partial \rev{x})^{-1}Q_k)Q_k^T\|_F^2=&\|\widehat{\Delta X}_k^{\mathcal I_k}-(\partial\bF/\partial \rev{x})^{-1}Q_k\|_F^2=\sum_{j=1}^{\widehat N}\|\widehat{\Delta X}_k^{\mathcal I_k}e_j-(\partial\bF/\partial \rev{x})^{-1}q_j\|^2\\
\leq&\, C^2 \kappa_F\rev{^2}(\rev{\bar R_k})\sum_{j=1}^{\widehat N}\left(\sum_{p=1}^j\ell_p\right)^2\\
\leq&\, C^2\varepsilon^2 \kappa_F^2(\rev{\bar R_k})\frac{\widehat N(\widehat N+1)(2\widehat N+1)}{6},
\end{align*}
where in the last step we used the identity $\sum_{j=1}^{\widehat N}j^2=\widehat N(\widehat N+1)(2\widehat N+1)/6$.
\end{proof}

An immediate consequence of Lemma~\ref{lemma5.7} \rev{and Lemma~\ref{lemma_kappa}} is that 
\begin{equation}
    \|B + (\Delta X_k^{\mathcal I_k}-B \Delta F_k^{\mathcal I_k})(\Delta F_k^{\mathcal I_k})^\dagger-(\partial\bF/\partial \rev{x})^{-1}\|\leq \|B-(\partial\bF/\partial \rev{x})^{-1}\|+C\rev{\bar\kappa(\tau,N)}\varepsilon\alpha\rev{(N)},
\end{equation}
as the spectral norm of a matrix is always less \rev{than} or equal \rev{to} its Frobenius norm.
This will be exploited in the next theorem to assess the convergence of~\eqref{eq:iteration_boostconv_stabilized}\rev{, which is the main result of this paper. In words, Theorem~\ref{Theorem_conv} states that if the fixed operator $B$ is a sufficiently accurate approximation of the inverse Jacobian at $x_*$, and the initial guess $x_0$ is sufficiently close to $x_*$, then robust {\tt BoostConv} converges linearly to $x_*$. The key hypothesis is the inequality~\eqref{eq:contraction_condition} below, which couples the three sources of error in the update~\eqref{eq:iteration_boostconv_stabilized}: the nonlinearity of $\bF$ (through $L\varepsilon$), the inaccuracy of $B$ as an approximation of the inverse Jacobian (through $\delta$), and the multisecant correction built upon the residual history (through $C\bar\kappa(\tau,N)\varepsilon\alpha(N)$). The latter term is specific to {\tt BoostConv} and it is kept under control by the rejection test of Algorithm~\ref{alg:stabilizedBoostConv}, through the uniform bound of Lemma~\ref{lemma_kappa}; it also shows that a large window $N$ or a small threshold $\tau$ require, in principle, a closer initial guess $x_0$. Notice that for a linear $\bF$ ($L=0$, hence $C=0$) the condition~\eqref{eq:contraction_condition} reduces to $\delta\rho<1$.}

\begin{theorem}\label{Theorem_conv}
Under the hypotheses and notation of Assumption~\ref{Assumption}, let $\{x_j\}_{j\geq1}$ be the sequence of iterates produced by~\eqref{eq:iteration_boostconv_stabilized} with $\xi_j$ computed by Algorithm~\ref{alg:stabilizedBoostConv} \rev{with threshold $\tau\in(0,1)$ and window $N$, and with $\xi_0=r_0$}. Then, assume there exist \rev{$\delta\geq0$ and $0<\varepsilon\leq\kappa$} such that
$$\|B -(\partial \bF/\partial x)^{-1}\|\leq \delta,\quad\text{and}\quad \|x_0-x_*\|\leq\varepsilon,$$
and
\rev{
\begin{equation}\label{eq:contraction_condition}
a:=\|(\partial\bF/\partial x)^{-1}\|L\varepsilon +(\delta+C\bar\kappa(\tau,N)\varepsilon\alpha(N))\rho<1,
\end{equation}
}
where $C$ \rev{is as in Lemma~\ref{lemma5.6}, $\bar\kappa(\tau,N)$ is as in Lemma~\ref{lemma_kappa},} and $\alpha\rev{(N)}$ \rev{is} as in Lemma~\ref{lemma5.7}. Then, 
\rev{for any $q\in[a,1)$ and all $j\geq0$, it holds}
$$x_{j+1}\in \rev{U_\varepsilon(x_*)\subseteq} E,\quad \text{and}\quad \|x_{j+1}-x_*\|\leq q \|x_{j}-x_*\|.$$
\end{theorem}

\begin{proof}
\rev{Since $\varepsilon\leq\kappa$, we have $U_\varepsilon(x_*) \subseteq U_\kappa(x_*) \subseteq E$. We proceed by induction on $j$, showing that if $x_0,\ldots,x_j\in U_\varepsilon(x_*)$, then $\|x_{j+1}-x_*\|\leq q\|x_j-x_*\|$, so that also $x_{j+1}\in U_\varepsilon(x_*)$.}

For $j=0$, \rev{since $\xi_0=r_0$,} we have
\begin{align*}
\|x_1-x_*\|=&\, \|x_0-x_*-B \bF(x_0)\|\\
\leq&\,\|x_0-x_*-(\partial\bF/\partial x)^{-1}(\bF(x_0)-\bF(x_*))\|+\|(\partial\bF/\partial x)^{-1}(\bF(x_0)-\bF(x_*))  
-B \bF(x_0)\| \\
\leq &\, \|(\partial\bF/\partial x)^{-1}\|\|(\partial\bF/\partial x)(x_0-x_*)-(\bF(x_0)-\bF(x_*))\|\\
&\,+\|(\partial\bF/\partial x)^{-1}-B\| \|\bF(x_0)-\bF(x_*)\|
\\
\leq&\, (\|(\partial\bF/\partial x)^{-1}\|L\varepsilon+\delta\rho) \|x_0-x_*\|\leq \rev{a\|x_0-x_*\|\leq} q\|x_0-x_*\|\leq \varepsilon,
\end{align*}
so that also $x_1\in U_\varepsilon(x_*)\subseteq E$.

Now\rev{, let $j\geq 1$ and} assume that \rev{$x_0,\ldots,x_j\in U_\varepsilon(x_*)\subseteq E$. In particular, all the iterates involved in the definition of $\Delta X_{j}^{\mathcal I_{j}}$ and $\Delta F_{j}^{\mathcal I_{j}}$ belong to $U_\varepsilon(x_*)$, so that Lemma~\ref{lemma5.7} applies with $\ell_p\leq\varepsilon$, whereas $\kappa_F(\bar R_j)\leq\bar\kappa(\tau,N)$ by Lemma~\ref{lemma_kappa} (if the retained history is empty, i.e., $\widehat N_j=0$, the term involving $\kappa_F(\bar R_j)$ below simply vanishes)}. We have
\begin{align*}
\|x_{j+1}-x_*\|=&\, \|x_{j}-x_*-B \bF(x_{j})- (\Delta X_{j}^{\mathcal I_{j}}-B \Delta F_{j}^{\mathcal I_{j}})(\Delta F_{j}^{\mathcal I_{j}})^\dagger \bF(x_{j})\|\\
\leq&\,\|x_{j}-x_*-(\partial\bF/\partial x)^{-1}(\bF(x_{j})-\bF(x_*))\|\\
&\, +\left(\|(\partial\bF/\partial x)^{-1}  
-B \| +\|(\widehat{\Delta X}_{j}^{\mathcal I_{j}}-(\partial\bF/\partial x)^{-1} Q_{j})Q_{j}^T\| 
\right)\|\bF(x_{j})-\bF(x_*)\|\\
\leq &\, \rev{\left(\|(\partial\bF/\partial x)^{-1}\|L\|x_j-x_*\|+(\delta+C\kappa_F(\bar R_j)\varepsilon\alpha(N))\rho\right)\|x_j-x_*\|}\\
\leq &\, \rev{\left(\|(\partial\bF/\partial x)^{-1}\|L\varepsilon +(\delta+C\bar\kappa(\tau,N)\varepsilon\alpha(N))\rho\right)\|x_j-x_*\|}\\
=&\, \rev{a\|x_j-x_*\|\leq  q\|x_j-x_*\|\leq\varepsilon,}
\end{align*}
\rev{where we used $\|x_j-x_*\|\leq\varepsilon$ and $\|\bF(x_{j})-\bF(x_*)\|\leq\rho\|x_j-x_*\|$. This concludes the induction and the proof.}
\end{proof}

We would like to underline the lack of any strong assumptions in Theorem~\ref{Theorem_conv}, in addition to Assumption~\ref{Assumption} and those stated in the theorem. In particular, writing $\bF$ in a fixed-point manner, namely $\bF(x)=x-\bG(x)$ for a certain \rev{nonlinear} $\bG$, we do not require the contractivity or nonexpansivity of the latter map\rev{, an assumption which is instead customary in the convergence analysis of Anderson acceleration; see, e.g.,~\cite{TothKelley2015,EvansEtAl2020}. On the other hand, we stress that the hypotheses of Theorem~\ref{Theorem_conv} do have consequences on the unaccelerated iteration~\eqref{eq:basic_iteration}. Indeed, denoting $J_*:=\partial\bF/\partial x(x_*)$ and recalling that $\|J_*\|\leq\rho$, they imply $\|I-BJ_*\|=\|(B-J_*^{-1})J_*\|\leq\delta\rho<1$, so that $B$ is necessarily nonsingular and the iteration~\eqref{eq:basic_iteration} is itself locally convergent to $x_*$. Theorem~\ref{Theorem_conv} thus guarantees that the residual recombination performed by robust {\tt BoostConv} does not spoil the local convergence of~\eqref{eq:basic_iteration}, and it provides the contraction factor $q$. The ability of {\tt BoostConv} to restore the convergence of a diverging base iteration, as observed in the numerical experiments of section~\ref{Numerical Results}, is not covered by the present analysis.} \DP{These types of results may be obtained by adapting the analysis presented in~\cite{Pollock_noncontractive} for (standard) Anderson to our multi-mixing setting. We leave this interesting research direction for future work.} 

Theorem~\ref{Theorem_conv} can be seen as a generalization of~\cite[Theorem 5.8]{brezinski_anderson} to the case $B\neq -\beta I$. Since in the literature about Anderson acceleration, the scalar $\beta$ is often called \emph{mixing parameter}, we can think of $B$ as a \emph{multi-mixing parameter}. We believe that many results scattered in the literature about Anderson acceleration can be generalized to the multi-mixing setting as well. In particular, we envision the possibility to show an improvement in the convergence of linearly-convergent fixed-point schemes thanks to our multi-mixing approach, as shown in~\cite{EvansEtAl2020,PollockRebholz2023,PollockRebholzBook} in the case of \rev{standard or filtered} Anderson acceleration. We leave this interesting research venue to be studied elsewhere.

\rev{We also recall that, as discussed in section~\ref{sec:broyden}, for a general $B$ {\tt BoostConv} does not coincide with Anderson acceleration applied to the ``preconditioned'' problem $1/\beta B\bF(x)=0$, as the two schemes solve different least squares problems. Hence, Theorem~\ref{Theorem_conv} cannot be obtained by simply applying~\cite[Theorem 5.8]{brezinski_anderson} to the preconditioned problem, which motivates the multi-mixing viewpoint adopted here.}

We conclude this section by mentioning that it may be possible to improve the local convergence result of Theorem~\ref{Theorem_conv} by introducing a step-length parameter $\eta_k\in\mathbb{R}$ so that~\eqref{eq:iteration_boostconv_stabilized} becomes
$$x_{k+1}=x_k+\eta_kB\xi_k.$$
The selection of $\eta_k$ could mimic what is done in classic approaches from the quasi-Newton literature; see, e.g.,~\cite[Chapter 6]{dennis_schnabel}.

\section{Numerical results}\label{Numerical Results}
In this section we present several numerical tests \rev{illustrating the performance of robust {\tt BoostConv}. We remark that these tests are not designed to verify quantitatively the hypotheses of Theorem~\ref{Theorem_conv}, which are local in nature and involve quantities, like $\delta$, $\varepsilon$, and $L$, that are hard to estimate in practice. The results reported below should thus be regarded as empirical evidence of the effectiveness of Algorithm~\ref{alg:stabilizedBoostConv}, also outside the regime covered by our analysis.} The \rev{MATLAB} code for reproducing the results in section~\ref{subsec:linear} and~\ref{subsec:burgers} can be found at \url{https://github.com/palittaUniBO/robustBoostConv/tree/main}.

%%%%%%%%%%%%%%%%%%%%%%%%%%%%%%%%%%%%%%%
{\color{black}
\subsection{A linear problem}
\label{subsec:linear}
We start by applying Algorithm~\ref{alg:stabilizedBoostConv} to the solution of a linear system of equations of the form $Ax=b$. As $A\in\mathbb{R}^{n\times n}$ and $b\in\mathbb{R}^n$ we consider the matrix {\tt fidap029} available in the Matrix Market\footnote{\url{https://math.nist.gov/MatrixMarket/data/SPARSKIT/fidap/fidap029.html}} and the related right-hand side {\tt fidap029\_rhs1} so that $n=2\,870$.

We choose two different stationary iterative methods for the solution of $Ax=b$. Each of these schemes can be written in the form~\eqref{eq:basic_iteration} for different $B$ and their convergence is driven by the spectral radius $\rho$ of the iteration matrix $I-BA$.
\begin{itemize}
    \item Richardson iteration: $B=I$. We have $\rho(I-BA)=0.9987$ so that we expect a rather slow convergence of this method;
    \item Jacobi method: $B=D^{-1}$ where $D$ denotes the diagonal matrix having on the diagonal the diagonal elements of $A$. Since $\rho(I-BA)=1.1350$ this method is expected to diverge.
    %\item TSVD: we select as $B$ the diagonal matrix whose first $1\,000$ diagonal elements are the reciprocal of the first $1\,000$ singular values of $A$ whereas all the other diagonal entries are zero. This means that this matrix $B$ is singular.
\end{itemize}

In Figure~\ref{fig:linear} we report the convergence history obtained by running up to 50 iterations of the basic iterative schemes~\eqref{eq:basic_iteration} (solid line) and the related counterpart~\eqref{eq:iteration2} enhanced by Algorithm~\ref{alg:stabilizedBoostConv} for $N=3$ and $\tau=10^{-10}$ (dotted line) for Richardson (left) and Jacobi (right). 

\begin{figure} \centering\includegraphics[width=\textwidth]{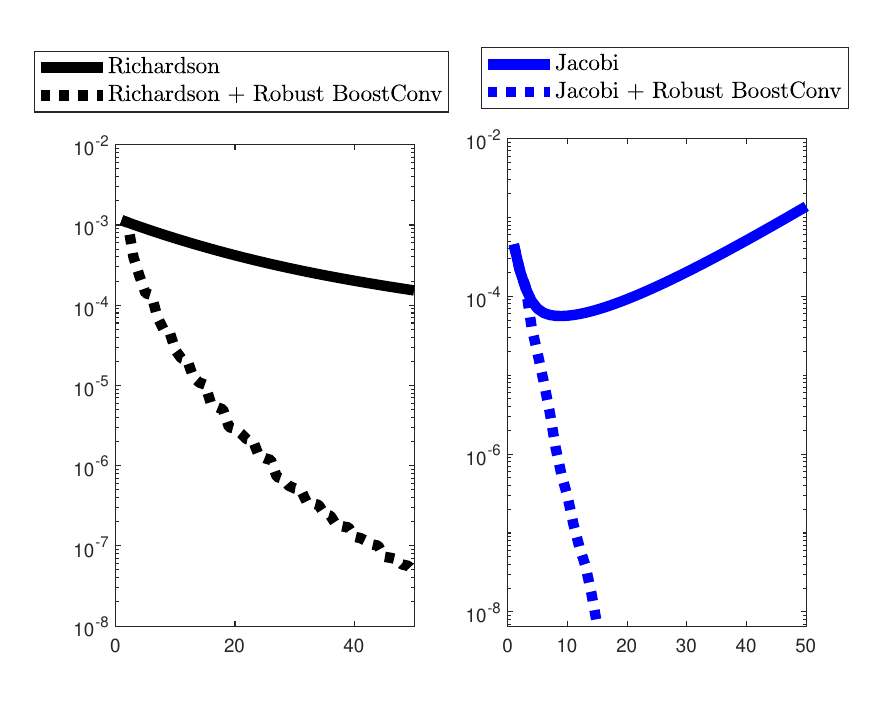} \caption{Relative residual norm history achieved by the plain version (solid line) of Richardson (left) and Jacobi (right) and their counterparts enhanced by Algorithm~\ref{alg:stabilizedBoostConv} with $N=3$ and $\tau=10^{-10}$ (dotted line) for the solution {\color{black}of} $Ax=b$. {\color{black}The horizontal axis reports the iteration index $k$, and the relative residual norm is $\|b-Ax_k\|_2/\|b-Ax_0\|_2$, with $x_0=0$.}}
\label{fig:linear} \end{figure}

\rev{Let us} start with the Richardson iteration. Since $B=I$ in this case, applying {\tt BoostConv} is equivalent to adopting the plain Anderson acceleration and the convergence of the Richardson iteration equipped with Anderson acceleration has been studied in~\cite{Pasini}. Figure~\ref{fig:linear} (left) confirms the ability of this approach in significantly boosting the convergence rate of the plain Richardson iteration. 

As we can see from Figure~\ref{fig:linear} (right), the Jacobi method is not really able to solve this linear system. Indeed, after a timid decrease of the residual norm in the first iterations, the latter starts growing dramatically. This was expected by looking at $\rho(I-BA)$. On the other hand, thanks to the use of Algorithm~\ref{alg:stabilizedBoostConv}, the residual norm provided by~\eqref{eq:iteration2} rapidly decreases getting below $10^{-8}$ in less than 20 iterations. In this case, {\tt BoostConv} is thus able to fully \rev{restore} the convergence of a diverging method, and not only to speed up its convergence rate. \rev{We stress that this example lies outside the regime covered by Theorem~\ref{Theorem_conv}. Indeed, for a linear $\bF$ we have $L=0$ and, as discussed after Theorem~\ref{Theorem_conv}, the hypotheses of the theorem imply $\|I-BA\|<1$, which is incompatible with a spectral radius of $I-BA$ larger than one. The convergence restored by {\tt BoostConv} in this case is thus an empirical finding that goes beyond our theoretical results.} \DP{As we already mentioned, by adapting the results in~\cite{Pollock_noncontractive} to our framework we may be able to explain the success of {\tt BoostConv} in retrieving convergence in this case.} 
We conclude by noting that the robust version of {\tt BoostConv} (Algorithm~\ref{alg:stabilizedBoostConv}) and its plain variant from~\cite{CITRO2017234} yield very similar results in this example, since a very small value of $N$ was used. For this reason, we report only the results obtained with the former.
}
%%%%%%%%%%%%%%%%%%%%%%%%%%%%%%%%%%%%%%%
\subsection{Convergence acceleration for the 1D Burgers' equation}
\label{subsec:burgers}

\begin{figure}  \centering\includegraphics[width=\textwidth]{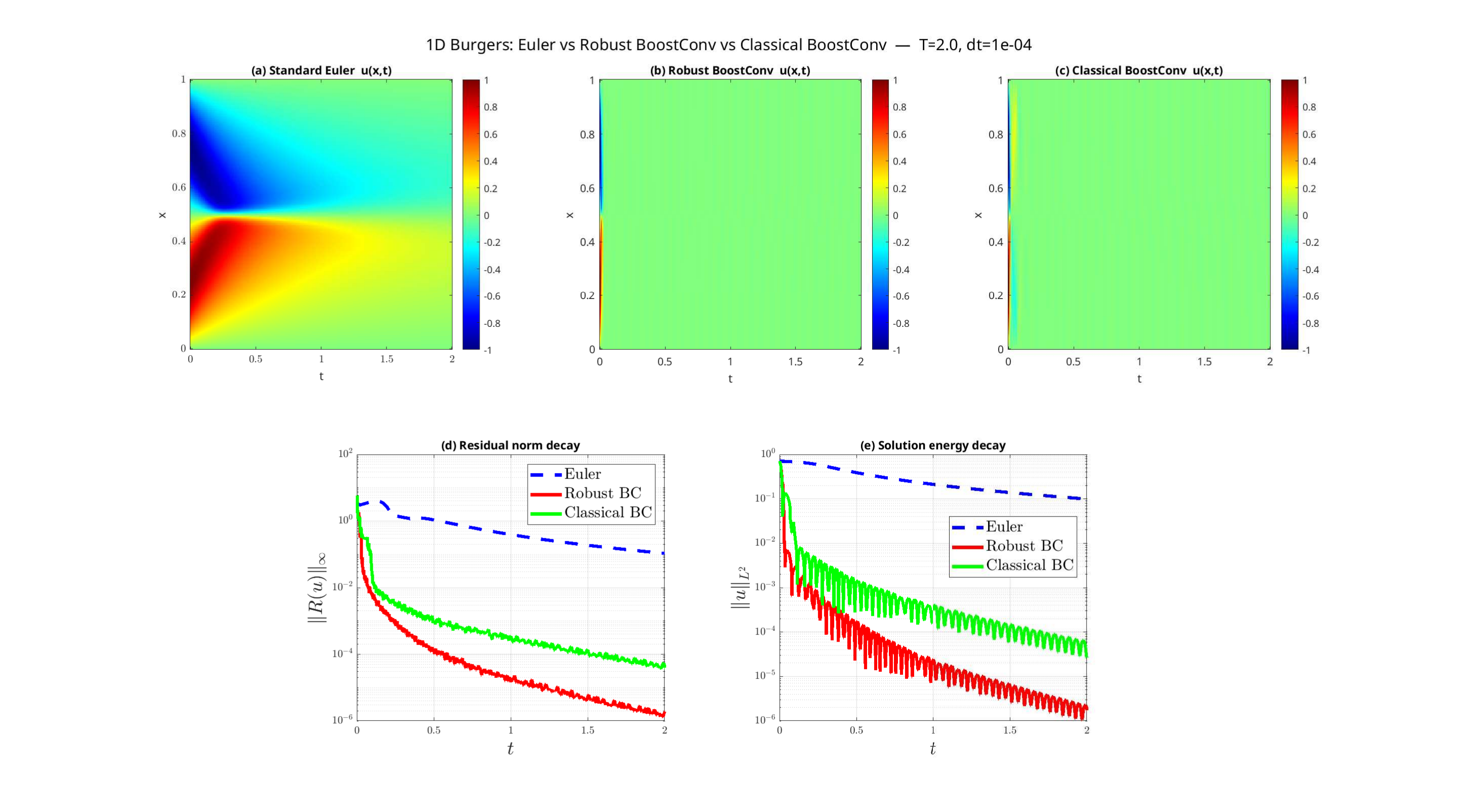} \caption{Comparison between standard Explicit Euler integration and {\tt BoostConv} acceleration for the 1D Burgers' equation. 
Panel \textbf{(a)} shows the temporal evolution of the solution $u(x,t)$ obtained using the standard Explicit Euler scheme. 
The system progressively relaxes toward the trivial steady solution $u(x,t)=0$. 
Panels \textbf{(b)} and \textbf{(c)} illustrate the effect of applying {\tt BoostConv} in its classical and robust variants, respectively, showing a substantial enhancement in the convergence toward the zero solution. 
To further quantify the performance of the two algorithms, panel \textbf{(d)} reports the decay of the residual norm $\|\mathcal{R}(u)\|_\infty$, while panel \textbf{(e)} shows the decay of the solution energy $\|u\|_{L^2}$. 
The results clearly indicate that the robust formulation provides a noticeable improvement in the convergence {\color{black}behavior. In panels \textbf{(b)}--\textbf{(e)}, the horizontal axis reports the pseudo-time $t=k\Delta t$, with $k$ the iteration index; for the accelerated schemes, $t$ is merely a rescaled iteration counter and does not carry any physical meaning (see the text).}}
\label{fig:burgers_combined} \end{figure}

We investigate the convergence properties of the proposed {\tt BoostConv} algorithm by considering the one–dimensional unsteady viscous Burgers' equation, which represents a prototypical nonlinear convective–diffusive model and is widely employed as a benchmark problem for testing numerical methods in fluid dynamics. The equation reads
{\color{black}
\[
\left\{
\begin{array}{
ll}
\partial_t u + u\,\partial_x u = \nu\,\partial_{xx} u,& \text{in }[0,1]\times (0,2),\\
     u(0,t)=u(1,t)=0,\\
     u(x,0)=\sin(2\pi x),
\end{array}
\right.
\]
where $u(x,t)$ denotes the velocity field and $\nu$ is the kinematic viscosity.

 This is a time-dependent problem and a popular approach to find a steady solution is to semi-discretiz\rev{e} the differential problem in space and apply a time-marching scheme. In particular,} the spatial domain $x\in[0,1]$ is discretized using a uniform grid composed of $N_x=200$ grid points, with second–order finite–difference approximations employed for both the convective and diffusive operators. The homogeneous Dirichlet boundary conditions, $u(0,t)=u(1,t)=0$, are enforced explicitly at each time step. 
Time integration is performed using an explicit Euler scheme with time step $\Delta t = 10^{-4}$ up to a final time $T=2$. The initial condition is prescribed as a smooth sinusoidal profile, $u(x,0)=\sin(2\pi x)$, which subsequently evolves under the combined action of nonlinear advection and viscous dissipation.

Three numerical strategies are considered: the baseline explicit Euler time--marching scheme, the classical {\tt BoostConv} acceleration~\cite{CITRO2017234} with $N=5$, and the robust {\tt BoostConv} variant proposed in the present work (Algorithm~\ref{alg:stabilizedBoostConv}) with $N=5$ and $\tau=10^{-10}$. 
{\color{black}
The Euler scheme can be seen as~\eqref{eq:basic_iteration} with $B=\Delta tI$ and the iteration index $k$ can be interpreted as the time step index.  
Even though {\tt BoostConv} directly acts on the right--hand side of the semi--discrete system, without modifying the underlying time integrator (cf.~\eqref{eq:iteration2}), its iteration index $k$ loses any connection with time evolution and simply represents the current iteration of our iterative method. \rev{For ease of comparison with the Euler integration, in Figure~\ref{fig:burgers_combined} all the results are nevertheless reported against the pseudo-time $t=k\Delta t$ which, for the accelerated schemes, is just a rescaled iteration counter. In particular, panels \textbf{(b)} and \textbf{(c)} of Figure~\ref{fig:burgers_combined} do not represent the physical time evolution of the Burgers\' equation, but rather the sequence of iterates produced by {\tt BoostConv}.} 
}

The results of this comparison are reported in \cref{fig:burgers_combined}. 
Panel \textbf{(a)} shows the spatio--temporal evolution of the velocity field obtained using the standard explicit Euler scheme. 
As expected, the solution relaxes progressively toward the trivial steady state $u(x,t)=0$, with the dynamics characterized by nonlinear steepening followed by viscous smoothing. 
Panels \textbf{(b)} and \textbf{(c)} illustrate the effect of applying {\tt BoostConv} in its classical and robust variant, respectively. 
In both cases the convergence toward the equilibrium state is significantly accelerated with respect to the baseline Euler integration, with the robust formulation exhibiting the most rapid stabilization of the solution.

This qualitative behavior is further quantified by the convergence histories reported in panels \textbf{(d)} and \textbf{(e)}. 
Panel \textbf{(d)} shows the evolution of the residual norm $\|\mathcal{R}(u)\|_\infty$, where $\mathcal{R}(u)$ denotes the semi--discrete Burgers operator. 
The standard explicit Euler method exhibits a slow and nearly exponential decay of the residual, reflecting the stiffness induced by the diffusive term and the unfavorable spectral properties of the linearized operator. 
The classical {\tt BoostConv} algorithm substantially accelerates the decay, while the robust variant further improves the convergence rate, allowing the residual to reach the prescribed tolerance significantly earlier.

A similar trend is observed for the solution energy, quantified by the $L^2$-norm $\|u\|_{L^2}$ and reported in panel \textbf{(e)}. 
While the standard scheme displays a gradual decay of the energy, both accelerated strategies induce a much faster relaxation toward the zero solution, with the robust formulation providing the most effective damping of the slowest modes of the system.

\rev{These observations are in line with the interpretation of {\tt BoostConv} as a multisecant method discussed in section~\ref{sec:broyden}.} 
The convergence rate of the explicit Euler scheme is governed by the spectral distribution of the Jacobian of the Burgers operator, whose eigenvalues cluster near the imaginary axis, resulting in slow decay rates for certain modes. 
By contrast, the {\tt BoostConv} methodology acts directly on the spectral content of the iteration through residual recombination, effectively clustering the dominant eigenvalues and suppressing slowly converging components. 
The robust formulation further improves this mechanism, leading to a more regular and efficient convergence behavior without altering the underlying time--integration scheme.

\subsection{Stabilization of Navier--Stokes equations}

% commentata altrimenti non compila. pero' la figura e' ok..

\begin{figure} 
\includegraphics[width=\textwidth]{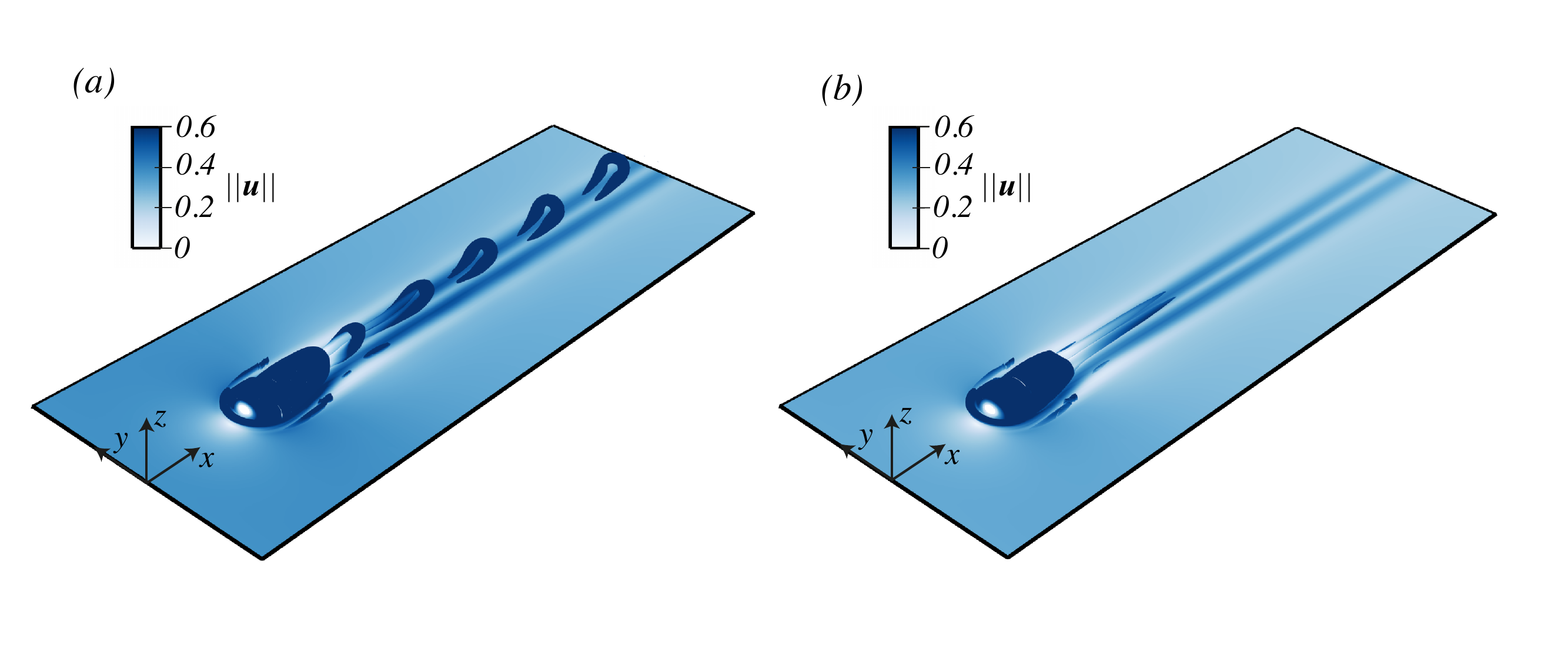}
\includegraphics[width=\textwidth]{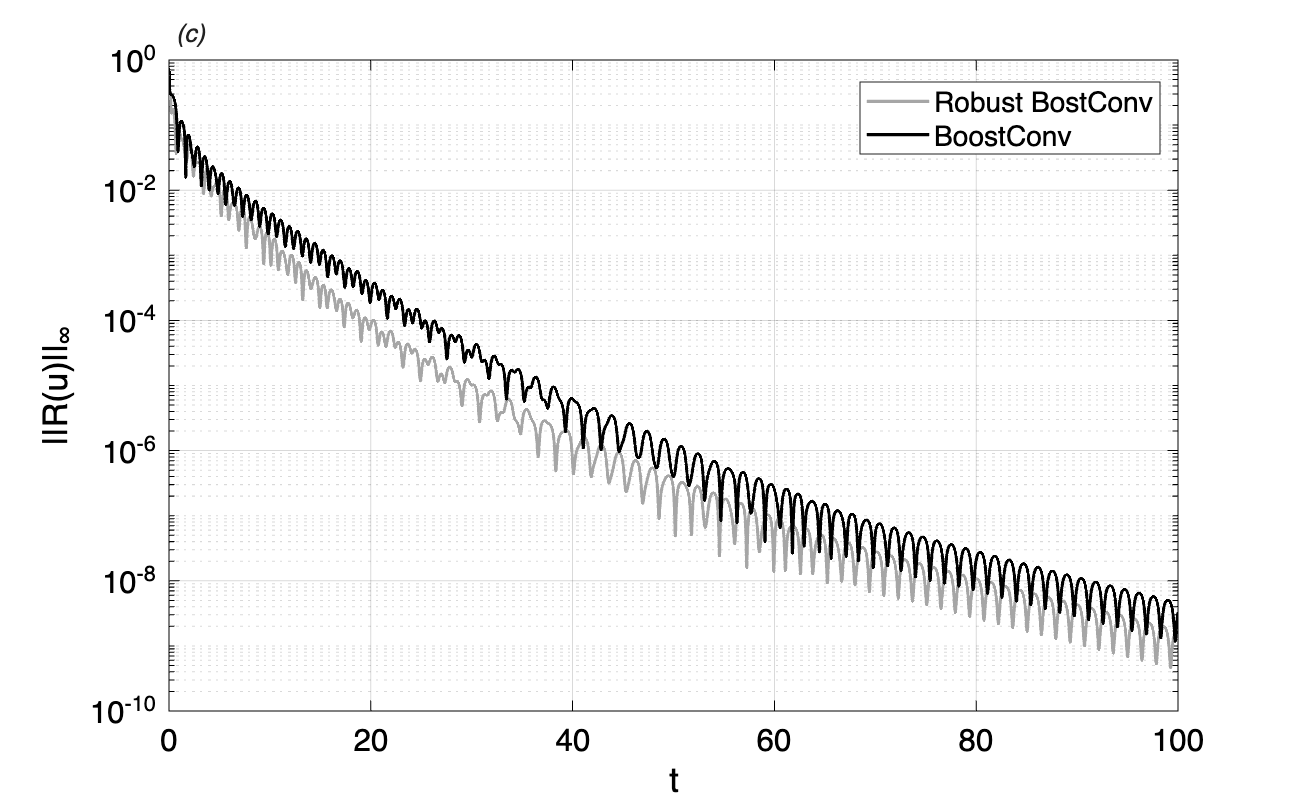}
\caption{Comparison between an unstable flow realization and the corresponding base flow, together with the stabilization {\color{black}behavior} of two algorithms. 
Panel \textbf{(a)} shows an instantaneous velocity field associated with an unstable solution, characterized by amplified perturbations and complex spatial structures. 
Panel \textbf{(b)} displays the base flow, obtained as the converged steady solution of the incompressible Navier--Stokes equations. 
Panel \textbf{(c)} illustrates the evolution of the stabilization process, comparing the {\color{black}behavior} of the original algorithm with that of the proposed robust version, highlighting the improved convergence. The BoostConv parameters are $N = 15$ and $\tau = 10^{-10}$ (threshold for linear dependency detection). 
Overall, the comparison emphasizes the qualitative differences between unstable dynamics and the underlying base state, as well as the enhanced stabilization properties of the robust algorithm. Parameter setting (see \cite{citroPOFBoostconv} for further details): $Re_k=500$, $k/\delta^*=2.5$.}
\label{fig:obstacle} 
\end{figure}

We next consider a high-dimensional test case governed by the incompressible Navier--Stokes equations, with the objective of assessing the performance of the proposed {\tt BoostConv} algorithm in a realistic large-scale flow simulation. In contrast to the low-dimensional benchmark problem discussed in the previous section, the present example reflects the intrinsic complexity of fully nonlinear flow computations, characterized by a very large number of degrees of freedom, strong nonlinear coupling, and the presence of weakly damped or unstable modes that critically influence the long-time convergence properties of iterative solvers. The primary objective is therefore not only to assess the stabilizing effect of the method, but also to demonstrate its capability to compute reference steady states that are otherwise difficult to obtain using conventional time-marching approaches.

We focus on the stability properties of viscous boundary-layer flows over a flat plate in the presence of a single hemispherical roughness element embedded in a Blasius boundary layer \cite{citroPOFBoostconv}. This configuration has been extensively investigated in the literature as a canonical model for roughness-induced transition, as it captures the essential physical mechanisms responsible for the amplification of disturbances, the onset of unsteadiness, and the formation of coherent vortical structures. As such, it provides a stringent and physically relevant test case for assessing the ability of numerical methods to resolve and stabilize complex flow dynamics.

The governing equations are the incompressible Navier--Stokes equations,
\[
\partial_t \mathbf{u} + \mathbf{u}\cdot\nabla \mathbf{u} + \nabla p
- \nu \Delta \mathbf{u} = \mathbf{0}, \qquad
\nabla\cdot\mathbf{u} = 0,
\]
where $\mathbf{u}$ denotes the velocity field, $p$ the pressure, and $\nu$ the kinematic viscosity. The equations are discretized in space using the spectral--element method as implemented in the open-source solver \texttt{Nek5000} (https://nek5000.mcs.anl.gov). High-order polynomial expansions are employed within each element, leading to a semi-discrete dynamical system of very large dimension, representative of state-of-the-art high-fidelity flow simulations. Time integration is performed using the standard semi-implicit scheme provided by \texttt{Nek5000}, in which the nonlinear convective term is treated explicitly, while the viscous contribution is handled implicitly.

For the flow configuration under consideration, there exists a critical stability threshold, param\-etrized by the Reynolds number, beyond which the steady boundary-layer solution becomes unstable. Above this threshold, time integration of the governing equations leads to a self-sustained unsteady flow characterized by oscillatory dynamics and periodic shedding of hairpin vortices downstream of the roughness element, as illustrated in \cref{fig:obstacle}~\textbf{(a)}. This unsteady regime dominates the long-time behavior of the system and prevents direct computation of the underlying steady solution through standard time-marching techniques.

Nevertheless, for the same set of parameters, the system admits a steady base flow that is linearly unstable. This base flow represents a fixed point of the Navier--Stokes equations and plays a central role in the theoretical analysis of roughness-induced transition, yet it is not directly accessible via conventional forward time integration due to its instability. The objective of this section is to demonstrate that {\tt BoostConv} enables the computation of this unstable base flow. The structure of the resulting fixed point is shown in \cref{fig:obstacle}~\textbf{(b)}, providing direct numerical evidence of the ability of {\tt BoostConv} to stabilize and converge to unstable steady solutions of high-dimensional nonlinear problems.
Further insight is provided by \cref{fig:obstacle}~\textbf{(c)}, where the stabilization histories obtained with the classical {\tt BoostConv} algorithm and with the robust variant proposed in the present work are directly compared. 
The results clearly show that the robust formulation enhances the stabilization process, improving convergence toward the fixed point.

%%%%%%%%%%%%%%%%%%%%%%%%%%%%%%%%%%%%%%%
\section{Conclusions}\label{Conclusions}

In this work we have presented a rigorous theoretical and computational investigation of the {\tt BoostConv} algorithm, providing for the first time a comprehensive mathematical framework that clarifies its interpretation and establishes convergence guarantees. By placing {\tt BoostConv} within the theory of residual recombination and Anderson--type acceleration methods, we have shown that the algorithm can be viewed as a multisecant strategy that constructs low--rank corrections to an underlying iteration operator while preserving a nonintrusive, black--box implementation.

From a theoretical standpoint, the main contribution of this study lies in the derivation of local convergence results under mild regularity assumptions on the nonlinear operator. The proposed robust formulation ensures that the matrices involved in the least--squares recombination remain well conditioned \rev{(up to a column scaling)}, thereby avoiding numerical degeneracies that typically arise in multisecant schemes. The resulting convergence theorem demonstrates linear convergence in a neighborhood of a regular solution without requiring contractivity of \rev{the fixed--point map $\bG$ in $\bF(x)=x-\bG(x)$, even though, as discussed in section~\ref{Convergence properties}, its hypotheses imply the local convergence of the unaccelerated iteration~\eqref{eq:basic_iteration}}. Moreover, the analysis highlights the role of the residual subspace in approximating the inverse Jacobian through a sequence of low--rank updates, providing a clear interpretation of {\tt BoostConv} as a multi--mixing generalization of Anderson acceleration.

The numerical experiments \rev{complement} the theoretical findings and illustrate the versatility of the proposed approach across problems of increasing complexity. {\color{black}We started by showing that {\tt BoostConv} is able to accelerate or even \rev{restore} the convergence of classic stationary iterative methods for the solution of linear systems of equations.} Then, 
in the one--dimensional Burgers' equation, {\tt BoostConv} effectively suppresses slowly decaying modes and dramatically accelerates convergence toward equilibrium \rev{without altering the underlying time--integration scheme}. More importantly, in the high--dimensional incompressible Navier--Stokes simulation, the method demonstrates its capability to stabilize nonlinear iterations and to compute unstable steady solutions that are inaccessible through conventional time--marching procedures. These results confirm that Algorithm~\ref{alg:stabilizedBoostConv} remains effective in large--scale settings characterized by non--normal operators, strongly coupled dynamics, and extremely large numbers of degrees of freedom.

A distinctive strength of {\tt BoostConv} is its minimal intrusiveness. Algorithm~\ref{alg:stabilizedBoostConv} operates solely at the algebraic level by recombining previously computed residuals, leaving the underlying discretization, linear solvers, and time--integration strategies unchanged. This property makes it particularly attractive for integration into existing high--performance simulation frameworks, where modifications to the numerical formulation are often impractical. Furthermore, the additional computational cost remains modest since the acceleration relies on a small least--squares problem and low--dimensional updates.

Beyond the specific results presented here, several avenues for future research emerge naturally from this work. On the theoretical side, extending the convergence analysis to global settings, \rev{to diverging base iterations, which {\tt BoostConv} is able to stabilize in practice, and to the quantification of the acceleration with respect to the base iteration~\eqref{eq:basic_iteration}, in the spirit of~\cite{EvansEtAl2020,PollockRebholz2023},} represents an important direction. The interpretation of the iteration matrix $B$ as a multi--mixing approximation of the inverse Jacobian also suggests potential connections with quasi--Newton and Krylov subspace methods that deserve further investigation. From an algorithmic perspective, adaptive strategies for selecting the recombination window $N$, periodic activation criteria, and line--search mechanisms may further enhance robustness in strongly nonlinear regimes. Finally, the application of {\tt BoostConv} to multiphysics problems, optimization, and large--scale inverse problems offers promising opportunities for future developments.

Overall, the results presented in this paper establish {\tt BoostConv} as a theoretically grounded and practically effective nonlinear acceleration strategy. By bridging the gap between empirical performance and rigorous analysis, this work provides a solid foundation for the systematic use of residual recombination techniques in large--scale scientific computing and opens new perspectives for the design of efficient, nonintrusive solvers for complex nonlinear systems.

\bibliographystyle{siamplain}
\bibliography{references}

@article{CITRO2017234,
title = {Efficient stabilization and acceleration of numerical simulation of fluid flows by residual recombination},
journal = {Journal of Computational Physics},
volume = {344},
pages = {234-246},
year = {2017},
doi = {10.1016/j.jcp.2017.04.081},
author = {V. Citro and P. Luchini and F. Giannetti and F. Auteri},
}

@article{Pollock_noncontractive,
    author = {Pollock, S. and Rebholz, L. G},
    title = {Anderson acceleration for contractive and noncontractive operators},
    journal = {IMA Journal of Numerical Analysis},
    volume = {41},
    number = {4},
    pages = {2841-2872},
    year = {2021},
    doi = {10.1093/imanum/draa095},
}

@article{Pasini,
author = {Lupo Pasini, Massimiliano},
title = {{Convergence analysis of Anderson-type acceleration of Richardson's iteration}},
journal = {Numerical Linear Algebra with Applications},
volume = {26},
number = {4},
pages = {e2241},
doi = {10.1002/nla.2241},
year = {2019}
}

@article{FangSaas2009,
author = {Fang, H.-r. and Saad, Y.},
title = {Two classes of multisecant methods for nonlinear acceleration},
journal = {Numerical Linear Algebra with Applications},
volume = {16},
number = {3},
pages = {197-221},
doi = {10.1002/nla.617},
year = {2009}
}

@Book{Golub2013,
  Title                    = {Matrix computations},
  Author                   = {Golub, G. H. and Van Loan, C. F.},
  Publisher                = {Johns Hopkins University Press, Baltimore, MD},
  Year                     = {2013},
  Edition                  = {Fourth},
  Series                   = {Johns Hopkins Studies in the Mathematical Sciences},

  ISBN                     = {978-1-4214-0794-4; 1-4214-0794-9; 978-1-4214-0859-0},
  Mrclass                  = {65-02 (65Fxx)},
  Mrnumber                 = {3024913},
  Mrreviewer               = {J\"org Liesen},
  Pages                    = {xiv+756}
}

@article{WalkerNi2011,
author = {Walker, H. F. and Ni, P.},
title = {Anderson Acceleration for Fixed-Point Iterations},
journal = {SIAM Journal on Numerical Analysis},
volume = {49},
number = {4},
pages = {1715-1735},
year = {2011},
doi = {10.1137/10078356X},
}

@article{EvansEtAl2020,
author = {Evans, C. and Pollock, S. and Rebholz, L. G. and Xiao, M.},
title = {{A Proof That Anderson Acceleration Improves the Convergence Rate in Linearly Converging Fixed-Point Methods (But Not in Those Converging Quadratically)}},
journal = {SIAM Journal on Numerical Analysis},
volume = {58},
number = {1},
pages = {788-810},
year = {2020},
doi = {10.1137/19M1245384},
}

@article{TothKelley2015,
author = {Toth, A. and Kelley, C. T.},
title = {{Convergence Analysis for Anderson Acceleration}},
journal = {SIAM Journal on Numerical Analysis},
volume = {53},
number = {2},
pages = {805-819},
year = {2015},
doi = {10.1137/130919398},
}

@book{dennis_schnabel,
  author    = {Dennis, J. E. and Schnabel, R. B.},
  title     = {Numerical Methods for Unconstrained Optimization and Nonlinear Equations},
  publisher = {SIAM},
  year      = {1996},
  address   = {Philadelphia}
}

@book{nocedal_wright,
  author    = {Nocedal, J. and Wright, S. J.},
  title     = {Numerical Optimization},
  edition   = {2nd},
  publisher = {Springer},
  year      = {2006},
  address   = {New York}
}

@article{broyden1965,
  author  = {Broyden, C. G.},
  title   = {A class of methods for solving nonlinear simultaneous equations},
  journal = {Math. Comp.},
  volume  = {19},
  year    = {1965},
  pages   = {577--593}
}

@book{kelley_book,
  author    = {Kelley, C. T.},
  title     = {Iterative Methods for Linear and Nonlinear Equations},
  publisher = {SIAM},
  year      = {1995},
  address   = {Philadelphia}
}

@article{anderson1965,
  author  = {Anderson, D. G.},
  title   = {Iterative procedures for nonlinear integral equations},
  journal = {J. ACM},
  volume  = {12},
  number  = {4},
  year    = {1965},
  pages   = {547--560}
}

@article{brezinski_anderson,
  author  = {Brezinski, C. and Cipolla, S. and Redivo-Zaglia, M. and Saad, Y.},
  title   = {{Shanks and Anderson-type acceleration techniques for systems of nonlinear equations}},
  journal = {IMA J. Numer. Anal.},
  volume  = {42},
  year    = {2022},
  pages   = {3058--3093}
}

@article{saad_szyld,
  author  = {Simoncini, V. and Szyld, D. B.},
  title   = {{Recent computational developments in Krylov subspace methods for linear systems}},
  journal = {Numer. Linear Algebra Appl.},
  volume  = {14},
  year    = {2007},
  pages   = {1--59}
}

@book{saad_iterative,
  author    = {Saad, Y.},
  title     = {Iterative Methods for Sparse Linear Systems},
  edition   = {2nd},
  publisher = {SIAM},
  year      = {2003},
  address   = {Philadelphia}
}

@article{citroPOFBoostconv,
  title={Global stability and sensitivity analysis of boundary-layer flows past a hemispherical roughness element},
  author={Citro, V and Giannetti, F and Luchini, P and Auteri, F},
  journal={Physics of Fluids},
  volume={27},
  number={8},
  pages={084110},
  year={2015},
  publisher={AIP Publishing LLC}
}

@article{PollockRebholz2023,
author = {Pollock, S. and Rebholz, L. G.},
title = {{Filtering for Anderson Acceleration}},
journal = {SIAM Journal on Scientific Computing},
volume = {45},
number = {4},
pages = {A1571--A1590},
year = {2023},
doi = {10.1137/22M1536741},
}

@book{PollockRebholzBook,
author = {Pollock, S. and Rebholz, L. G.},
title = {{Anderson Acceleration for Numerical PDEs}},
publisher = {SIAM},
address = {Philadelphia, PA},
year = {2025},
doi = {10.1137/1.9781611978490},
}

\end{document}